\documentclass[12pt]{amsart}
\usepackage{amssymb,amsmath,enumerate,xcolor}

\newtheorem{thm}{Theorem}
\newtheorem{lem}[thm]{Lemma}
\newtheorem{prop}[thm]{Proposition}

\newtheorem*{conj*}{Conjecture}
\newtheorem{cor}[thm]{Corollary}

\newcommand{\N}{\mathbb N}
\newcommand{\Z}{\mathbb Z}
\newcommand{\C}{\mathbb C}
\newcommand{\PP}{\mathbb P}
\newcommand{\R}{\mathbb{R}}

\newcommand{\diag}{\mathcal D}

\newcommand{\diagpar}[1]{\mathcal D_{#1}}
\newcommand{\DELETE}[1]{}

\newcommand{\ep}{\epsilon}
\newcommand{\del}{\delta}
\newcommand{\om}{\omega}
\newcommand{\oom}{{\bar\omega}}
\newcommand{\Om}{\Omega}
\newcommand{\lam}{\lambda}

\newcommand{\sig}{\sigma}

\newcommand{\topspacemat}[2]{E_{#1}(#2)}

\newcommand{\topspaceDSunpert}[2]{E_{#1}(#2)}
\newcommand{\topspaceDSpert}[2]{E_{#1}^\epsilon(#2)}

\newcommand{\bottomspacemat}[2]{F_{#1}(#2)}

\newcommand{\bottomspaceDSunpert}[2]{F_{#1}(#2)}
\newcommand{\bottomspaceDSpert}[2]{F_{#1}^\epsilon(#2)}

\newcommand{\cocycle}[2]{A^{(#1)}_{#2}}

\newcommand{\pertA}{A^\epsilon}
\newcommand{\cocyclepss}[2]{{\pertA}_{#2}^{(#1)}}

\DeclareMathOperator{\Cov}{Cov}
\DeclareMathOperator{\vol}{Vol}
\DeclareMathOperator{\LPS}{\mathcal E^\text{L}}
\DeclareMathOperator{\RPS}{\mathcal E^\text{R}}
\DeclareMathOperator{\PS}{\mathcal E}
\DeclareMathOperator{\E}{\mathbb E}
\DeclareMathOperator{\real}{Re}
\DeclareMathOperator{\lin}{lin}
\newcommand{\HS}{\mathsf{HS}}
\newcommand{\SHS}{\mathsf{SHS}}
\newcommand{\bad}{\mathsf{Filler}}
\newcommand{\op}{\mathsf{op}}
\newcommand{\nbhd}{\mathcal U}

\newtheorem{thmx}{Theorem}

\begin{document}
\title[Hilbert Space Lyapunov Exponent stability]
{Hilbert Space Lyapunov Exponent stability}
\author{Gary Froyland}
\email{g.froyland@unsw.edu.au}
\author{Cecilia Gonz\'alez-Tokman}
\email{cecilia.gt@uq.edu.au}
\author{Anthony Quas}
\email{aquas@uvic.ca}
\date{\today}

\begin{abstract}
We study cocycles of compact operators acting on a separable Hilbert space,
and investigate the stability of the Lyapunov exponents and Oseledets
spaces when the operators are subjected to additive Gaussian noise.
We show that as the noise is shrunk to 0, the Lyapunov exponents of the
perturbed cocycle converge to those of the unperturbed cocycle; and the
Oseledets spaces converge in probability to those of the unperturbed
cocycle. This is, to our knowledge, the first result of this type with
cocycles taking values in operators on infinite-dimensional spaces. The
infinite dimensionality gives rise to a number of substantial difficulties
that are not present in the finite-dimensional case.
\end{abstract}

\maketitle

\section{Introduction}
A question of paramount importance in applied mathematics is:
How to tell if the conclusions derived from a model indeed capture
relevant features of an underlying  system? Stability results address
this question by giving conditions under which small changes in a
model entail small changes in the outcomes of the analysis.

In the last decade, multiplicative ergodic theory has been developed in
the so-called semi-invertible setting (that is the setting in which the
underlying base dynamics are assumed to be invertible, but no invertibility
assumptions are made on the matrices) \cite{FLQ1, FLQ2, GTQ1,GTQ2} with
the aim of providing a useful mathematical tool to analyse transport
features of complex real world systems, such as geophysical flows.
This approach has been implemented to find coherent structures in fluid flow 
\cite{FLS10}, and a finite-time version of this theory has been used to detect 
atmospheric vortices and oceanic eddies in geophysical flows 
\cite{FroylandSantiMonahan,FHRSS12}.

However, from the mathematical perspective, the following  questions
remain completely unsolved:
\begin{itemize}
\item \textsl{Model or data errors:} Do these structures -- obtained using either models 
of geophysical flows or observational data, both of which contain errors -- correspond to 
\textit{real} features of the underlying flows?
\item \textsl{Numerical errors:} Are these structures robust to numerical errors in the 
numerical schemes applied to the models or observational data in order to extract the
ergodic-theoretic objects?
\end{itemize}

The aim of this work is to provide an initial step in establishing
conditions for the stability of Lyapunov exponents and so-called
Oseledets spaces, the essential components underlying multiplicative
ergodic theory, in an infinite-dimensional context.
The infinite dimensionality aspect is crucial to be able to eventually
encompass the setting of \textit{transfer operators} -- a powerful
mathematical tool used to model transport in dynamical systems.
In the infinite-dimensional context, and aside from works focusing exclusively on the i.i.d.~
perturbation (noise) setting, stability results have only been
established either (i) under uniform hyperbolicity assumptions on the underlying cocycle,
which for example cover the case of random perturbations of a fixed map
\cite{BaladiKondahSchmitt,Bogenschutz};
or (ii) for the top (first) component of the splitting, in the context of
transfer operators \cite{FGTQ-nonlinearity}, where the leading
Lyapunov exponent is always 0, corresponding to a random fixed point.

Early results concerning stability of Lyapunov exponents for
finite-dimensional (matrix) cocycles include
\cite{Ruelle79,Kifer82,KiferSlud82,Hennion84}.
In the setting of invertible matrix cocycles, the closest results
to this work are due to Ledrappier, Young and Ochs
\cite{Young86, LedrappierYoung, Ochs}.
The difficulty of the stability problem at hand, even in the
finite-dimensional setting, is highlighted by the existence of
negative stability results for Lyapunov exponents of matrix cocycles
\cite{Bochi,BochiViana}, which show that for non-uniformly hyperbolic cocycles,
carefully chosen arbitrarily small perturbations may collapse the
entire spectrum of Lyapunov exponents to a single exponent.
In this finite-dimensional setting, the stability problem
remains an active topic of research, and related recent
results include \cite{BockerViana,BlumenthalXueYoung}.
In the setting of semi-invertible matrix cocycles, the authors
established stability results under stochastic perturbations in
\cite{FGTQ-cpam, FGT}.

In this paper, we study cocycles taking values in compact operators on
a separable Hilbert space.
The unperturbed cocycle is assumed
to be strongly coercive, with exponentially-decaying transmission
between higher order modes, so that the leading Oseledets spaces
tend to be concentrated on low order modes.
This issue of the cocycle sending an arbitrarily high-order mode to a low-order 
mode does not arise in the finite-dimensional setting.
Additionally, unlike the finite-dimensional case, there is no natural Lebesgue-like 
measure on the infinite-dimensional space of perturbations.
Hence as our model of noise we use additive Gaussian perturbations.
The Gaussian nature of the perturbations allows for
unbounded changes, and is also convenient for calculations.
In order to maintain the noise as a \emph{small} perturbation, the Gaussian 
perturbations are required to have stronger exponential decay than the unperturbed
cocycle.
We regard the model as a natural generalisation of the finite-dimensional 
Ledrappier-Young setting to infinite dimensions.

The main results of the paper, Theorems~\ref{thm:exponents} and
\ref{thm:spaces}, yield, respectively, convergence of Lyapunov
exponents and Oseledets spaces of the randomly perturbed cocycles.
The method of proof of stability of Lyapunov exponents 
builds on the work of Ledrappier and Young
\cite{LedrappierYoung}, which dealt with Lyapunov exponents in
invertible matrix cocycles, as well as on our recent work
\cite{FGTQ-cpam}, which had to handle the complications arising from
non-invertibility of the matrices. The motivation for studying the case
of non-invertible matrices is that transfer operators are generally not invertible.
The strategies in all three papers, \cite{LedrappierYoung, FGTQ-cpam}
and this one, are similar in spirit:
The idea is to split long sequences of matrices observed along the
cocycle into \emph{good} and \emph{bad} blocks, depending on
whether or not the long term behaviour of the cocycle corresponds
to the observed behaviour within the block, and then handle
carefully the concatenations.
However, at the technical level,  there are substantial
complications arising from the need to handle \emph{wild}
perturbations occurring in possibly higher and higher modes.

As in the previous works \cite{Ochs,FGTQ-cpam}, the stability of Oseledets spaces 
is deduced from the stability of the Lyapunov exponents,
but the strategy of the proof here is different.
The approach of Ochs \cite{Ochs} applies only to invertible matrices, 
and the proof is essentially finite-dimensional. 
The core of the argument is:
if the perturbed slowest Oseledets spaces were often far from its unperturbed counterpart,
the contribution
to the bottom exponent of the perturbed system on this part of the 
base space would be at least $\lambda_{d-1}$. 
Hence, convergence of the exponents implies the perturbed and unperturbed 
slow spaces are mostly nearby. This is basically an expectation argument.
Subsequent Oseledets spaces are similarly controlled using exterior powers.
The approach of \cite{FGTQ-cpam} in the context of not necessarily invertible matrices
relies on the use of M\"obius transformations or graph transforms. 
The essence of the argument is one fixes all of the perturbations to the matrices other
than the perturbation at time $-1$. Since there is exponential contraction in a cone around the
unperturbed \emph{fast space} (that is the span of the $k$-dimensional Oseledets 
spaces with largest Lyapunov exponents), all but a very small set of perturbations 
at time $-1$ cause
the fast space to fall into the basin of attraction, and to end up near the unperturbed
fast space. 
While this approach would still apply in the infinite-dimensional case, 
the new argument of this paper
has the advantages that it is simpler and more general; in particular, it does not rely on any
special structure for the perturbations, such as absolute continuity, which played a role in 
\cite{FGTQ-cpam}. 
All that is required is that the perturbations are small with high probability.
The approach in the current paper goes as follows: if the 
perturbed $k$-dimensional fast space is not close to the  
unperturbed fast space at time $2N$ (where $N$ is the block size), 
then the minimum angle between the perturbed fast space
at time $N$ and the unperturbed slow space at time $N$ must be small. 
For this to happen, the
minimum angle between the perturbed fast space at time 0 and the unperturbed slow space at
time 0 has to be exponentially small. Whenever this happens, there is a
growth drop of the $k$-dimensional volume of order $\exp(-(\lambda_k-\lambda_{k+1})N)$
over this block. 
An expectation argument ensures that this must happen rarely because otherwise the perturbed
$\lambda_k$ would be much less than the unperturbed $\lambda_k$. 

\section{The model and principal results}\label{S:Model}
Throughout the paper $\sigma\colon\Omega\to\Omega$ is an invertible measurable
transformation, $\PP$ is an ergodic invariant measure, and $H$ is a separable 
Hilbert space with basis $e_1,e_2,\ldots$.

The Hilbert-Schmidt norm is $\|A\|_\HS^2=
\sum_{i,j}\langle Ae_i,e_j\rangle^2$.
Define a stronger norm: $\|A\|_\SHS^2=
\sum_{i,j}2^{2(i+j)}\langle Ae_i,e_j\rangle^2$.
We frequently think of operators with bounded HS norm as
infinite matrices where the entries are square-summable.
We write $\HS$ for the collection of Hilbert-Schmidt operators on $H$ (those
operators, $A$, satisfying $\|A\|_\HS<\infty$),
and $\SHS$ for the collection of strong Hilbert-Schmidt operators (those operators
satisfying $\|A\|_\SHS<\infty$).

We write $\cocycle n\om$ for the unperturbed cocycle: $\cocycle n\om=
A_{\sigma^{n-1}\omega}\cdots A_\omega$, and call $A\colon\Omega\to\SHS$
the \emph{generator} of the operator cocycle.
Throughout the article, $\Delta$ will denote the random Hilbert-Schmidt operator
with independent normal entries with mean 0 and where the $(i,j)$ entry has
standard deviation $3^{-(i+j)}$. Write $\gamma$ for the measure on $\SHS$
corresponding to this distribution.
We apply a sequence of independent perturbations
$\mathbf\Delta=(\Delta_n)_{n\in\Z}$, where each $\Delta_n$ has the distribution above. 
For $\om$ lying in the base space, we denote by $\oom$ the pair
$(\om,\mathbf\Delta)$ specifying the point of the base space and the
sequence of perturbations. The space of such pairs is denoted by
$\bar\Omega$, and is equipped with the transformation $\bar\sigma=\sigma\times s$,
where $s$ is the left shift on the sequence of perturbations and the 
ergodic invariant measure $\bar\PP=\PP\times\gamma^\Z$.
The perturbed cocycle is parameterized
by $\epsilon$ (a measure of the size of the perturbation) and defined by
$$
\cocyclepss n\oom=(A_{\sigma^{n-1}\om}+\epsilon\Delta_{n-1})\cdots
(A_\om+\epsilon\Delta_0).
$$

\begin{thmx}\label{thm:exponents}
Let $\sigma\colon\Omega\to\Omega$ be an invertible measurable transformation
and let $\PP$ be an ergodic invariant measure for $\sigma$.
Let $H$ be a separable Hilbert space and let $A\colon\Omega\to\SHS$
be the generator of an operator cocycle satisfying $\int\log\|A_\omega\|_\SHS\,
d\PP(\omega)<\infty$.

Let $\bar\Omega$, $\bar\sigma$ and $\bar\PP$ be as defined above.
For each parameter $\epsilon>0$, define a new cocycle $\pertA\colon\bar\Omega
\to\SHS$ over $\bar\sigma$ 
with generator $\pertA(\oom)=A(\omega)+\epsilon\Delta_0$.
Then the Lyapunov exponents of $\pertA$ (listed with multiplicity)
converge to those of $A$
as $\epsilon\to 0$.
\end{thmx}

\begin{thmx}\label{thm:spaces}
Assume the hypotheses and notation of Theorem~\ref{thm:exponents}.
Let the (at most countably many) distinct Lyapunov exponents of the  cocycle $A$ be
$\lambda_1>\lam_2> \ldots >-\infty$, with corresponding multiplicities
$d_1,d_2,\ldots$.  Let the corresponding Oseledets decomposition be
$\SHS=Y_1(\omega) \oplus Y_2(\omega) \oplus\ldots$.
Let $D_0=0$, $D_i=d_1+\ldots+d_i$ and let the Lyapunov exponents (with multiplicity)
be $\infty>\mu_1\ge \mu_2\ge\ldots > -\infty$, so that $\mu_j=\lambda_i$
if $D_{i-1}<j\le D_i$.

Let $\nbhd_i= (\lambda_i-\alpha, \lambda_i+\alpha)$ be a neighbourhood of
$\lambda_i$ not containing any other exponent of the unperturbed cocycle.
Let $\epsilon_0$ be such
that for each $\epsilon\le \epsilon_0$ and each $D_{i-1}<j\le D_i$, the
$j^{\text{th}}$ Lyapunov exponent $\mu_j^\ep$ of the perturbed cocycle satisfies
$\mu_j^\ep\in \nbhd_i$.
For $\epsilon<\epsilon_0$, let $Y^\ep_i(\bar\omega)$ denote the sum of the Oseledets
subspaces of $\pertA$ having exponents in $\nbhd_i$.
Then $Y^\ep_i(\bar \omega)$ converges in probability to $Y_i(\omega)$ as $\ep\to 0$.
\end{thmx}

For $\lambda>1$, we let $\diagpar\lambda$ be the diagonal matrix whose
$(i,i)$ entry is $\lambda^{-i}$.  Formally we can write the random operator 
$\Delta$ from Theorem \ref{thm:exponents} as
$\diagpar3N\diagpar3$, where $N$ is a countably infinite square matrix of
independent standard normal random variables.

Throughout the remainder of the paper there will be numerous constants. We will mostly
just use the symbol $C$ to indicate a constant, where $C$ may refer
to different constants
at different places, even in the same proof. That is, whenever we write $C$,
we refer to a
quantity that may depend on $k$ (the number of exponents that we aim to control), and
on the underlying dynamical system, but not on $\epsilon$,
the size of the perturbations.
The exception to this will be some of the principal propositions where estimates are
collected for assembly in Section \ref{sec:puttingitalltogether}. In these propositions,
constants will be numbered according to the proposition in which they are found, so that
$C_{\ref{lem:trivial}}$, for example, is defined in Lemma \ref{lem:trivial}.

We briefly describe the structure of the proof of Theorem \ref{thm:exponents}
since there is considerable preparation before we start the proof.
The bulk of the proof is concerned with giving a lower bound for the
sum of the $k$ leading perturbed exponents, that is the maximal logarithmic growth
rate of $k$-volumes.
Given $\epsilon$, one defines a block length, $N\sim|\log\ep|$. For a large $n$, we
estimate the top exponents of the product $\cocyclepss{nN}\oom$, a perturbed block of
length $nN$. First, we replace the (sub-additive) logarithmic $k$-volume growth,
$\Xi_k(\cdot)$ by a
related approximately super-additive quantity, $\tilde\Xi_k(\cdot)$
(Sections \ref{sec:splitting} and \ref{sec:Comparison}). We use this
super-additivity to split $\cocyclepss{nN}\oom$
into good super-blocks (of length a multiple of $N$) and bad blocks (of length $N-2$),
that is $\Xi_k(\cocyclepss{nN}\oom)\gtrsim \tilde\Xi_k(\cocyclepss{nN}\oom)
\gtrsim \sum\tilde \Xi_k(\text{blocks})$. In section \ref{sec:goodBlocks}, ingredients
for the
estimate $\Xi_k(G^\ep)\gtrsim \Xi_k(G)$ are established,
where $G$ represents a good super-block and $G^\ep$ its perturbed version.
In sections \ref{sec:badPertI} and \ref{sec:badPertII},
ingredients for $\tilde \Xi_k(B^\ep) \gtrsim \tilde \Xi_k(B)$
are established (where $B$ is a bad block and $B^\ep$ is its
perturbed version). The estimates $\tilde\Xi_k(B)\gtrsim\Xi_k(B)$
and $\tilde\Xi_k(G^\ep)\gtrsim\Xi_k(G^\ep)$ are based on ingredients in
Section \ref{sec:Comparison}. Re-assembling the pieces using
sub-additivity of $\Xi_k$ and accounting for the errors gives the
result.

\section{Notation and the quantity $\tilde\Xi_k$}\label{S:Notation}

Recall that the Grassmannian of a Banach space is the space of closed
complemented subspaces. In a Hilbert space, every closed subspace is
complemented (by its orthogonal complement). We define $\mathcal G_k(H)$
to be the space of (necessarily closed) $k$-dimensional subspaces of $H$ and
$\mathcal G^k(H)$ to be the space of closed $k$-codimensional subspaces of $H$.
The collection of all closed subspaces of $H$ will be written $\mathcal G(H)$.
We will reserve the symbol $S$ for the unit sphere of $H$ throughout the article.

We define a metric on $\mathcal G(H)$ by
$$
\angle(U,V)=\max\left(\max_{u\in U\cap S}\min_{v\in V\cap S}\|u-v\|,
\max_{v\in V\cap S}\min_{u\in U\cap S}\|u-v\|\right),
$$
that is the Hausdorff distance between the intersections of the two subspaces
with the unit sphere. We remark that this differs by at most a bounded factor from 
another metric, the `gap' between closed subspaces defined in 
Kato \cite{Kato}. This is a complete metric on $\mathcal G(H)$.

We also make use of a measure of transversality between two subspaces of
complementary dimensions:
if $U\in \mathcal G^k(H)$ and $V\in \mathcal G_k(H)$,
then
$$
{\perp}(U,V)=\frac 1{\sqrt 2}\min_{u\in U\cap S,v\in V\cap S}\|u-v\|.
$$
The normalization is chosen so that if $U$ and $V$ have a common vector, then ${\perp}(U,V)=0$,
while if they are orthogonal complements, then ${\perp}(U,V)=1$.
We have the reverse triangle inequality: if $U',U\in \mathcal G^k(H)$, $V,V'\in\mathcal G_k(H)$,
then ${\perp}(U',V')\ge {\perp}(U,V)-\angle(V,V')-\angle(U,U')$.

We already introduced the classes of linear operators $\HS$ and $\SHS$ on $H$
with their associated norms, so that we have $\SHS\subset \HS\subset K(H)$,
where $K(H)$ stands for the compact linear operators on $H$. We write
$\|\cdot\|_\op$ for the operator norm, so that $\|\cdot\|_\SHS\ge \|\cdot\|_\HS$
for elements of $\SHS$ and $\|\cdot\|_\HS\ge\|\cdot\|_\op$ for elements of $\HS$.

For compact operators on $H$, the notions of singular vectors and singular values
pass directly from the finite-dimensional case. If $A\in K(H)$, we write
$s_1(A)\ge s_2(A)\ge\ldots$ for the singular values (with multiplicity in
decreasing order). The maximal logarithmic rate of $k$-dimensional volume growth is given by
$\Xi_k(A):=\log(s_1(A)\cdots s_k(A))$.

Define
$$
\tilde\Xi_k(A)=\E\Xi_k(\Pi_k\Delta A\Delta'\Pi_k),
$$
where $\Pi_k$ denotes orthogonal projection onto the
subspace of $H$ spanned by $e_1,\ldots,e_k$ and $\Delta$ and $\Delta'$
are independent copies of the random Hilbert-Schmidt operator. The key reason
for the introduction of $\tilde\Xi_k$ is that it satisfies an approximate
\emph{super-additivity} property (see Proposition \ref{prop:splitting})
that complements the sub-additivity of $\Xi_k$.

We denote by $\bar\Omega$, the space $\Omega\times \SHS^\Z$ and act on $\bar\Omega$
with the transformation $\sigma\times s$, where $s$ is the left-shift map on
$\SHS^\Z$. The space $\bar\Omega$ is equipped with the measure $\PP\times \gamma^\Z$,
where $\gamma$ is the multi-variate normal distribution on $\SHS$ described above
in which distinct elements of $\Delta$ are independent and the $(i,j)$ element
is normal with mean 0 and variance $3^{-2(i+j)}$.
We write $\bar\omega$ for a typical element of $\bar\Omega$,
that is a pair $(\omega,\mathbf\Delta)$, where $\mathbf \Delta=(\Delta_n)_{n\in\Z}$.

Informally, we expect an inequality like $\tilde\Xi_k(A)\ge
\Xi_k(A)-\LPS_k(A)-\RPS_k(A)$.
By $\LPS_k(A)$ (which stands for `left energy'), we
mean a measure of the modes on which the top $k$ left singular
vectors are distributed, while $\RPS_k(A)$ measures the modes
where the right singular vectors are supported. For example,
if the top left singular vectors are $e_7$, $e_8$, $e_{11}$ and $e_{13}$,
we expect $\LPS_4(A)$ to be approximately 39\,$\log 3$.

\begin{lem}\label{lem:simdiag}
Let $V$ be a $k$-dimensional subspace of $H$.
Let $D$ be a bounded operator on $H$.
There exists an orthonormal basis $v_1,\ldots,v_k$
for $V$ with the property that $Dv_1,\ldots,Dv_k$ are
mutually orthogonal.
\end{lem}

This follows from the singular value decomposition of finite-dimensional
operators.

\begin{lem}\label{lem:symmcloseness}
Let $U$ and $V$ be $k$-dimensional subspaces of $H$.
Then the two quantities appearing in the definition of $\angle(U,V)$ are equal:
$$
\max_{u\in U\cap S}\min_{v\in V\cap S}\|u-v\|=
\max_{v\in V\cap S}\min_{u\in U\cap S}\|u-v\|.
$$
\end{lem}

\begin{proof}
Let $\Pi_U$ be the orthogonal projection onto $U$ and $\Pi_V$ be the orthogonal
projection onto $V$. Then the singular vectors of $\Pi_V\circ\Pi_U$ give an orthogonal
basis of $U$, $u_1,\ldots,u_n$ with images $s_1v_1,\ldots,s_nv_n$, where
$v_1,\ldots,v_n$ form an orthogonal basis of $V$ (if $\Pi_V\Pi_Uu_i=0$, then
$v_i$ can be chosen to be an arbitrary unit vector of $V$ satisfying the orthogonality
condition). Write $u_i=s_iv_i+w_i$ with $w_i\in V^\perp$. One can then check that
$\langle u_i,v_j\rangle=0$ if $i\ne j$. Notice that $u_i$ and $s_iv_i$ are either equal or
non-collinear. It follows from the above that $U+V$ may be expressed as the
orthogonal direct sum $\lin\{u_1,v_1\}\oplus
\ldots\oplus\lin\{u_n,v_n\}$. One can now check that the linear map $R$ from $U+V$ to itself
mapping $u_i$ to $v_i$ and vice versa is an isometry interchanging $U$ and $V$. 
Applying this map yields the desired equality.
\end{proof}

Let $V$ be a $k$-dimensional subspace of $H$, and
$\Pi$ be the orthogonal projection onto $V$.
We define the \emph{energy} of $\Pi$ (also
the `energy of $V$') to be
$$
\PS_k(\Pi)=-\Xi_k(\Pi\circ \diagpar3)=-\sum_{i=1}^k\log\|\diagpar3v_i\|,
$$
where the $(v_i)$ are as guaranteed by the Lemma \ref{lem:simdiag}
with the operator $D$ taken to be $\diagpar3$.

\begin{lem}\label{lem:PS}
For any $k\in\N$, there exists a $C>0$ such that
if $\Pi$ and $\Pi'$ are orthogonal projections onto $k$-dimensional
subspaces and $Q\subset\HS$ satisfies $\gamma(Q)\ge\frac 12$, then
$$
\big|\E\big(\Xi_k(\Pi\Delta\Pi')\big|\Delta\in Q\big)+
(\PS_k(\Pi)+\PS_k(\Pi'))\big|\le C.
$$
\end{lem}

\begin{proof}
Let $u_1,\ldots,u_k$ be the basis guaranteed by Lemma
\ref{lem:simdiag} (applied with $D=\diagpar3$) for the range of $\Pi$
and $v_1,\ldots,v_k$ be the corresponding basis for $\Pi'$.

Now $\Xi_k(\Pi\Delta\Pi')=\log\det|M|$, where $M$ is a random matrix whose
$(i,j)$ entry is $\langle u_i,\Delta v_j\rangle$. The entries of $M$
therefore have a multi-variate normal distribution. Each has mean 0, so
the unconditioned distribution of $M$ is determined by the covariance of the pairs
of entries of the matrix.

Using the fact that the coordinates of the $u$'s and $v$'s are bounded
and the entries of $\Delta$ decay exponentially, we calculate
\begin{align*}
\Cov(M_{ij},M_{i'j'})
&=
\E \sum_{l,m,l',m'}(u_i)_l\Delta_{lm}(v_j)_m
(u_{i'})_{l'}\Delta_{l'm'}(v_{j'})_{m'}\\
&=
\sum_{l,m}3^{-2(l+m)}(u_i)_l(u_{i'})_l(v_j)_m(v_{j'})_m\\
&=\langle \diagpar 3u_i,\diagpar3u_{i'}\rangle
\langle \diagpar3v_j,\diagpar3v_{j'}\rangle,
\end{align*}
where for the second line, we used the fact that distinct entries of $\Delta$
are independent, and so have 0 covariance.
We see then, by the choice of $u$'s and $v$'s, that distinct entries of $M$ have 0 covariance, and
so are independent. The variance of the $(i,j)$ entry of the matrix is
$\|\diagpar3u_i\|^2\|\diagpar3v_j\|^2$, hence the unconditioned distribution of the
$(i,j)$ entry of the matrix is $\|\diagpar3u_i\|\|\diagpar3v_j\|$
times a standard normal.

Notice that the entire $i$ row has a multiplicative factor of
$\|\diagpar3u_i\|$ and the entire $j$ column has a multiplicative
factor of $\|\diagpar3v_j\|$, so that the determinant is
$\prod_i\|\diagpar3u_i\|\prod_j\|\diagpar3v_j\|\det(N_k)$, where
$N_k$ is a $k\times k$ random matrix with independent standard
normal entries, so that taking logarithms, we see
$\Xi_k(\Pi\Delta\Pi')=-\PS_k(\Pi)-\PS_k(\Pi')+\log|\det N_k|$.

Replacing $\Delta$ with a conditioned version has the effect of
multiplying the density of $N_k$ by a factor in the range $[0,2]$.
Since $\log|\det(N_k)|$ is an integrable function, there are uniform upper and lower
bounds for $\int\log|\det(N_k)|\rho(N_k)$ over all functions $\rho$ taking values in
$[0,2]$, so that
$$
\big|\E\big(\Xi_k(\Pi\Delta\Pi')\big|\Delta\in Q\big)+
\PS_k(\Pi)+\PS_k(\Pi')\big|\le C,
$$
as required.
\end{proof}

\begin{cor}\label{cor:PS}
There exists $K>0$ such that
if $\Pi'$ and $\Pi''$ are two orthogonal projections and $Q\subset \HS$
satisfies $\gamma(Q)>\frac 12$, then
$$
\Big|
\E\big(\Xi_k(\Pi\Delta\Pi')|\Delta\in Q\big) -
\big(\E\Xi_k(\Pi\Delta\Pi_k) + \E\Xi_k(\Pi_k\Delta\Pi')\big)\Big| < C
$$
\end{cor}

\begin{proof}
By Lemma \ref{lem:PS}, we have the following
\begin{align*}
&\big|\E\big(\Xi_k(\Pi\Delta\Pi')\big|\Delta\in Q\big)+
\PS_k(\Pi)+\PS_k(\Pi')\big|\le C;\\
&\big|\E\Xi_k(\Pi\Delta\Pi_k)+
\PS_k(\Pi)+\PS_k(\Pi_k)\big|\le C;\\
&\big|\E\Xi_k(\Pi_k\Delta\Pi')+
\PS_k(\Pi)+\PS_k(\Pi_k)\big|\le C,
\end{align*}
where $C$ is the constant from Lemma \ref{lem:PS}.

We calculate that $\PS_k(\Pi_k)=\frac 12k(k-1)\log 3$, so that
combining the inequalities, we obtain
$$
\big|\E\big(\Xi_k(\Pi\Delta\Pi')\big|\Delta\in Q\big)
-\big(\E\Xi_k(\Pi\Delta\Pi_k) + \E\Xi_k(\Pi_k\Delta\Pi')\big)\Big|
\le K,
$$
where $K=3C+k(k-1)\log 3$.
\end{proof}
%

\section{Good Blocks}\label{sec:goodBlocks}
\label{S:step1}
This section deals with \textsl{good blocks}.  The strategy we follow
goes back to Ledrappier and Young in the context of invertible
matrices \cite[Lemmas 3.3, 3.6 \& 4.3]{LedrappierYoung}, and it was
later used in \cite{FGTQ-cpam}.  Lemma~\ref{lem:ly} is the main tool
to control the effect of perturbations on good blocks.
Lemma~\ref{lem:goodBlocks} collects standard facts about Lyapunov
exponents, Oseledets splittings and their approximations via singular
vectors, which are used to define good blocks.
Lemma~\ref{lem:goodPert} establishes the conditions defining
\textsl{tame perturbations}.  Proposition \ref{prop:GlueingGoodBlocks}
provides a lower bound on $\Xi_k$ over a sequence of tame
perturbations, comparable with $\Xi_k$ for the unperturbed cocycle.

For each $k\in \N$, we define $\topspacemat kA$ to be the space
spanned by the \emph{images} of the singular vectors with $k$ largest singular values
under $A$, and $\bottomspacemat kA$ to be the space spanned by the
orthogonal complement of the pre-image of $\topspacemat kA$ under $A$.
Thus, $\bottomspacemat kA$ is exactly the space spanned by those
singular vectors of $A$ whose singular value is not amongst the $k$ largest.
We note that the spaces
$\bottomspacemat kA, \topspacemat kA$ are uniquely defined when the
singular values $s_k(A)$ and $s_{k+1}(A)$ are distinct. We will
always use our results in this setting, and therefore do not worry
about the possibility of non-uniqueness.

We collect some properties of singular values and singular vectors for compact
operators on Hilbert spaces and matrices.
\begin{lem}
\label{lem:sv}
Let $A$ be a  compact operator on a Hilbert space, $H$. Let the singular values be
$s_1(A),s_2(A),\ldots$.
\begin{enumerate}[(a)]
\item
$s_j(A)=\min_{V\in\mathcal G^{j-1}(H)}\max_{x\in V\cap S}\|Ax\|$;\label{it:minmax}

\item
$s_j(A)=\max_{V\in\mathcal G_j(H)}\min_{x\in V\cap S}\|Ax\|$;\label{it:maxmin}
\item\label{it:sjcont}
$|s_j(A)-s_j(B)|\le\|A-B\|_\op$;\label{it:cty}
\end{enumerate}
\end{lem}

\begin{proof}
The characterizations \eqref{it:minmax} and \eqref{it:maxmin} are well known.

To show \eqref{it:cty}, using \eqref{it:maxmin}, let $V$ be a $j$-dimensional space
such that $\|Ax\|\ge s_j(A)$ for all $x\in V\cap S$. Then $\|Bx\|\ge s_j(A)-\|A-B\|_\op$
for all $x\in V\cap S$, so that using \eqref{it:maxmin} again, we see
$s_j(B)\ge s_j(A)-\|A-B\|_\op$. By symmetry, $s_j(A)\ge s_j(B)-\|A-B\|_\op$, giving
the result.
\end{proof}

\begin{lem}
\label{lem:perp}
Let $U\in\mathcal G^k(H)$ and $V\in \mathcal{G}_k(H)$.
Then $s_k(\Pi_{U^\perp}\Pi_V)\ge{\perp}(U,V)$.
\end{lem}

\begin{proof}
Choose $v\in V$ with $\|v\|=1$.
Let $v=u+w$ with $u\in U$ and $w\in U^\perp$. Let $\hat u\in U\cap S$
be such that $u=\|u\|\hat u$ ($\hat u$ may be chosen arbitrarily if $u=0$)
and let
$\theta$ be the angle between $\hat u$ and $v$, so that $0<\theta\le\frac\pi2$.
By assumption $\|\hat u-v\|\ge\sqrt 2\,{\perp}(U,V)$.
We have $\|\hat u-v\|=2\sin\frac \theta2$.
Notice that $\|w\|=\|\Pi_{U^\perp}v\|=\sin\theta=2\sin\frac\theta2\cos
\frac\theta2\ge \sqrt 2\,{\perp}(U,V)\cos\frac\theta2$. Since $\theta\le\frac\pi 2$,
we see $\|\Pi_{U^\perp}v\|\ge {\perp}(U,V)$ for all $v\in V\cap S$.
\end{proof}

\begin{lem}\label{lem:ly}
For any $\delta<\frac 12$, there exists a $K>\delta^{-(4k+3)}$ such that if (i)  the
$k$th singular value of a compact linear operator $A: X \to X$ exceeds $K$; (ii)
the $(k+1)$st singular value of $A$ is at most 1; and (iii) $\|B-A\|\le 1$, then the
following hold:
\begin{enumerate}[(a)]
\item\label{it:pertsvs}
   $e^{-\del}\le s_j(A)/s_j(B)\le e^\del$ for each $j\le k$
    and $s_j(B)\le 2$ for each $j>k$;
\item
  $\angle (\topspacemat kA,\topspacemat kB)$ and
  $\angle(\bottomspacemat kA,\bottomspacemat kB)$ are less than
    $\delta$;\label{it:spacecont}
\item If $V$ is any subspace of dimension $k$ such that
  ${\perp}(V, \bottomspacemat kA)>\delta$, then
  $\angle(BV,\topspacemat kA)<\delta$;\label{it:contr}
\item If $V$ is a subspace of dimension $k$ and
    ${\perp}(V,\bottomspacemat kA)>2\delta$, then $|\det(B|_V)|\ge
   \delta^{k}\exp\Xi_{k}(B)$.\label{it:det}
  \end{enumerate}
\end{lem}

\begin{proof}
  For each closed subspace $W$ of $X$, let $\Pi_W:X \to W$ be the
  orthogonal projection onto $W$.
  \begin{enumerate}[(a)]
  \item
For the first part, notice that by assumption, for $j\le k$, we have
$s_j(A)\ge K$.  Also by Lemma \ref{lem:sv}\eqref{it:sjcont}, we have
$|s_j(A)-s_j(B)|\le 1$, so that $\frac{K}{K+1}\le s_j(A)/s_j(B)\le \frac{K}{K-1}$.
The second part of the claim follows from Lemma \ref{lem:sv}\eqref{it:sjcont} also.

\item Let $K>1+\frac{6}{\del}$. For symmetry, in this part, we assume only
$s_k(A),s_k(B)\ge K-1$, $s_{k+1}(A),s_{k+1}(B)\le 2$ and $\|A-B\|_\op\le 1$.

Let $v\in S$ satisfy $d(v,\bottomspacemat kA\cap S)\ge\delta$.
We will show that $v\not\in F_k(B)$. Let $v=u+w$ with $u\in F_k(A)$ and $w\in F_k(A)^\perp$.
By assumption, $\|w\|\ge\frac\delta2$, so that $\|Bv\|\ge \|Av\|-1\ge \|Aw\|-1
>(K-1)\|w\|-1>2$. On the other hand, if $v\in \bottomspacemat kB$, then
$\|Bv\|\le s_{k+1}(B)\le 2$.
The identical argument shows that if $v\in \bottomspacemat kA\cap S$, then
$d(v,\bottomspacemat kB\cap S)<\delta$

To show the closeness of the fast spaces, first let $v\in \bottomspacemat kB^\perp\cap S$,
and write $v$ as $au+w$, where $u\in \bottomspacemat kA\cap S$ and
$w\in \bottomspacemat kA^\perp$. Let $u'\in \bottomspacemat kB\cap S$
satisfy $\|u-u'\|<\delta$ (such a $u'$ exists by the paragraph above). Now $\langle v,u\rangle
=\langle v,u'\rangle + \langle v,u-u'\rangle$. The first term is 0 and the second term
is less than $\delta$ in absolute value. Hence $|a|<\delta$ and $\|w\|\ge \frac 12$.
Now $Bv=aAu+Aw+(B-A)v$. In particular, $\|Bv-Aw\|\le 2\delta+1\le 2$ while $\|Bv\|\ge
K-1$. Hence if $z\in\topspacemat kB\cap S$, we have $d(z,\topspacemat kA)\le 2/(K-1)$,
so $d(z,\topspacemat kA\cap S)\le 4/(K-1)$. The identical argument holds if the roles
of $A$ and $B$ are reversed, so $\angle(\topspacemat kA,\topspacemat kB)<4/(K-1)<\delta$.

\item
Let $K>4/\delta^2+2/\delta$.
Let $v\in V\cap S$ and write $v=u+w$ with $u\in\bottomspacemat kA$
and $w\in\bottomspacemat kA^\perp$. By assumption, $\|w\|\ge\delta$.
Hence $\|Aw\|\ge K\delta$, while $\|Au\|\le 1$. Since $\|B-A\|_\op\le 1$,
we have $\|Bv-Aw\|\le \|Bv-Av\|+\|Av-Aw\|\le 2$, so that $\|Bv-Aw\|/\|Bv\|
\le 2/(K\delta-2)$. Hence for an arbitrary element, $y$ of $BV\cap S$, we
have $d(y,\topspacemat kA)\le 2/(K\delta-2)<\frac\delta2$ and
$d(y,\topspacemat kA\cap S)\le 4/(K\delta-2)<\delta$. By Lemma
\ref{lem:symmcloseness}, we deduce that
$\angle(BV,\topspacemat kA)<\delta$ as required.

\item We have that $ \log |\det(B|_V)|\geq \Xi_k(\Pi_{\topspacemat kB}B|_V)=
\Xi_k(B\Pi_{\bottomspacemat kB^\perp} |_V)=
\Xi_k(B\Pi_{\bottomspacemat kB^\perp}\Pi_V)
=\Xi_k(B \Pi_{\bottomspacemat kB^\perp}) +
\Xi_k(\Pi_{\bottomspacemat kB^\perp} \Pi_V)\geq \Xi_{k}(B) + k
\log \delta.$ The last inequality follows from the facts that
$\Xi_k(B \Pi_{\bottomspacemat kB^\perp}) =\Xi_{k}(B)$; and
${\perp}(\bottomspacemat kB^\perp,
V)\ge {\perp}(\bottomspacemat kA^\perp,V)-\angle(\bottomspacemat kA^\perp,
\bottomspacemat kB^\perp)>\delta$
so that
$\| \Pi_{\bottomspacemat kB^\perp} \Pi_V v \| \geq
\delta \|v\|$ for every $v \in V$ by Lemma \ref{lem:perp}, hence
$\Xi_k(\Pi_{\bottomspacemat kB^\perp} \Pi_V)\geq k \log
\del$. The claim follows.

\end{enumerate}
\end{proof}

The following lemma underlies the definition of good blocks:
Using the notation of the lemma,
if $n\geq n_0$ and $\om \in G$, and  we say the block $\cocycle n\om$ is \emph{good}.
See \cite[Lemma 2.4]{FGTQ-cpam} for a proof in the context of matrix cocycles,
which applies without changes in our setting.

\begin{lem}[Good blocks]\label{lem:goodBlocks}
Let $\sigma$ be an invertible
ergodic measure-preserving transformation of $(\Omega,\PP)$ and
let $A\colon\Omega\to \SHS$ be a measurable map, taking values
in the strong Hilbert-Schmidt operators on $H$, and such that $\int\log^+\|A(\om)\|_\SHS
\,d\PP(\om)<\infty$. Let the Lyapunov exponents of the cocycle $A$
be $\infty>\mu_1\geq \mu_2\geq \ldots \ge-\infty$, counted
with multiplicities. Suppose $k \geq 1$ is such that
$\mu_k>0>\mu_{k+1}$. Let $\topspaceDSunpert k\om$ and
$\bottomspaceDSunpert k\om$ denote the
$k$-dimensional and $k$-codimensional Oseledets spaces of $A$
at $\om$ corresponding to Lyapunov exponents
$\mu_1\geq \dots \geq \mu_k$ and $\mu_{k+1}\geq \dots$, respectively.

Let $\xi>0$ and $\delta_1>0$ be given. Then there exist $n_0>0$,
$\tau\le\min(\delta_1,\frac 14\mu_k)$ and
$0<\delta\le\delta_1$ such that: for all
$n\ge n_0$, there exists a set $G\subseteq \Omega$ with $\mathbb
P(G)>1-\xi$ such that for $\omega\in G$, we have

\begin{enumerate}[(a)]
\item
$\perp(\bottomspaceDSunpert k\omega,\topspaceDSunpert k\omega)>
10\delta$;\label{it:sep}
\item
$\angle(\topspacemat k{\cocycle n{\omega}},
\topspaceDSunpert k{\sigma^n\omega})<\delta$;\label{it:Econt}
\item
$\angle(\bottomspacemat k{\cocycle n\omega},\bottomspaceDSunpert k\omega
)<\delta$;\label{it:Fcont}
\item
  $e^{(\mu_k+\tau)n}>s_{k}(\cocycle n\omega)>\max(K(\delta),e^{(\mu_k-\tau)n})$ and
  $s_{k+1}(\cocycle n\omega)<1$,
  where $K(\delta)$
  is as given in Lemma \ref{lem:ly}\label{it:svsep}.
\item $\frac{1}{n}\sum_{i=0}^{n-1}\log(1+\|A_{\sigma^i\om}\|_\SHS)<
2\int\log(1+\|A_\om\|_\SHS)\,d\PP(\om)$.\label{it:logSHS}
\end{enumerate}
\end{lem}

Assume that $\epsilon>0$ is fixed. A perturbation $\Delta$ is said to be \emph{tame} if
$|\Delta_{s,t}|\le \epsilon^{-1/2}(\frac 23)^{s+t}$ for all $s,t$
(otherwise $\Delta$ is \emph{wild}). A quick calculation
shows that if $\Delta$ is tame, then $\|\ep\Delta\|_\HS<2\sqrt\ep$.

\begin{lem}[Good block length]\label{lem:goodPert}
  Let $\sigma: (\Omega,\PP) \circlearrowleft$ be an ergodic
  measure-preserving transformation.  Let $A: \Omega\to \mathcal
  B(H)$ be a measurable map, taking values in the bounded linear
  operators on $H$, such that $\log^+\|A(\om)\|_\op$ is integrable. There
  exists $C_{\ref{lem:goodPert}}>0$ such that for all $\eta_0>0$,
  there exists $\ep_0$ such
  that for all $\epsilon<\epsilon_0$, there exists $G\subseteq\Omega$
  of measure at least $1-\eta_0$ such that for all $\omega\in G$, if
  $(\Delta_n)\in \HS^\Z$ satisfies $\Delta_n$ is tame
  for each $0\le n<N$, then
\begin{equation*}
\|\cocyclepss N{\bar\om}-\cocycle N\om\|_\op\le1,
\end{equation*}
where $\bar\om=(\om,(\Delta_n))$, $N=\lfloor C_{\ref{lem:goodPert}}|\log\ep|\rfloor$ and
 $\cocyclepss N{\bar\om}= \pertA_{N-1}(\bar\omega)\dots
 \pertA_1(\bar\omega) \pertA_0(\bar\omega)$.

The probability that one of $\Delta_0,\ldots,\Delta_{N-1}$
is wild is $O(e^{-1/(2\epsilon)})$.
\end{lem}

\begin{proof}
Let $g(\om)=\log^+(\|A_\om\|_\op+1)$ and let $C>0$ satisfy $\int g(\om)
\,d\PP(\om)<1/(2C)$.
Notice that provided $\epsilon<\frac14$ (and assuming that the perturbations
$(\Delta_n)_{0\le n<N}$ are tame,
so that $\|\epsilon\Delta_n\|_\op\le \|\epsilon\Delta_n\|_\HS\le
2\sqrt\ep$ for $0\le n<N$),
$\log^+ \|A_{\bar\sigma^n\bar\om }^\ep\|\leq g(\sigma^n\om)$
for each $0\le n<N$, and
\begin{align*}
\|\cocyclepss N\oom-\cocycle N\om\|_\op&\le\sum_{i=0}^{N-1}
\|\cocyclepss {N-i-1}{\bar\sigma^{i}\oom}
(\pertA_{\bar\sigma^i\oom}-A_{\sigma^i\om})
\cocycle i\om\|_\op\\
&\le 2N\sqrt\ep\exp(g(\om)+\ldots+g(\sigma^{N-1}\om)).
\end{align*}
There exists $n_0$ such that for $N\ge n_0$,
$g(\om)+\ldots+g(\sig^{N-1}\om)\le N/(2C)-\log(4N)$
on a set of measure at least $1-\eta_0$, hence
$2N\sqrt\ep\exp(g(\om)+\ldots+g(\sigma^{N-1}\om))
\le \frac 12\sqrt\ep \exp(N/(2C))$ on a set of measure at least $1-\eta_0$.
In particular, provided $\lfloor C|\log\ep_0|\rfloor>n_0$,
taking $N=\lfloor C|\log\ep|\rfloor$, we have
$\|\cocyclepss N{\bar\om}-\cocycle N\om\|_\op\le1$ provided that the
perturbations $\Delta_0,\ldots\Delta_{N-1}$ are all tame.

Recall that $(i,j)$th entry of $\Delta$ is distributed as $3^{-(i+j)}$ times
a standard normal random variable. Hence the probability that
$|\Delta_{i,j}|>\epsilon^{-1/2}(\frac23)^{i+j}$ is $\PP(|N|>\epsilon^{-1/2}2^{i+j})$.
Using a standard estimate on the tail of a normal random variable
\cite[Theorem 1.2.3]{Durrett},
this is at most $\frac{2\sqrt\epsilon}{\sqrt{2\pi}}
2^{-(i+j)}\exp(-2^{2i+2j-1}/\epsilon)$.

In particular, using the union bound, the probability that one
of $\Delta_0,\ldots,\Delta_{N-1}$ is wild is $O(e^{-1/(2\epsilon)})$.
\end{proof}

We comment that once $\xi>0$ and $\delta_1>0$ are fixed, Lemma \ref{lem:goodBlocks}
guarantees the existence of an $n_0$ such that for all sufficiently large $n$,
the good set defined in the lemma has measure at least $1-\xi$. Now
for $\epsilon$ sufficiently small, the length $N=\lfloor
C_{\ref{lem:goodPert}}|\log\epsilon|\rfloor$ exceeds $n_0$. For the remainder of the
proof, we let $G$ be the good set from Lemma \ref{lem:goodBlocks} with $n$ taken
to be $N$ (so that the good set, $G$, depends on $\xi$, $\delta_1$ and $\epsilon$,
but this dependence will not be made explicit). We further introduce the notation
$\bar G=G\cap \bigcap_{i=0}^{N-1}\{\Delta_i\text{ is tame}\}$, which we shall also
use for the remainder of the proof.

\begin{prop}[Glueing good blocks]\label{prop:GlueingGoodBlocks}
Under the assumptions of Lemma~\ref{lem:goodBlocks},
suppose $j<l$ and $\bar\sigma^{jN}\bar\omega,\bar\sigma^{(j+1)N}\bar\omega,
\ldots,\bar\sigma^{(l-1)N}\bar\omega\in \bar G$.
 Then,
\begin{equation}\label{eq:goodblox}
  \Xi_k(\cocyclepss{(l-j)N}\oom) \ge \Xi_k(\cocycle{(l-j)N}\om)+2(l-j)k\log\delta.
\end{equation}
\end{prop}

\begin{proof}
Let $B_n=\cocycle N{\sigma^{nN}\omega}$ and $\tilde B_n=\cocyclepss N{\bar\sigma^{nN}\oom}$.
This is proved by induction using Lemma~\ref{lem:ly}. Recall that since
$B_n$ is a good block, $\|B_n - \tilde{B}_n\|\leq 1$ by Lemma \ref{lem:goodPert}. We let
$\tilde V_j=V_j=\bottomspacemat k{B_j}^\perp$ and define $V_{n+1}=B_nV_n$ and
$\tilde V_{n+1}=\tilde B_n\tilde V_n$.

We claim that the following hold, for each $n=j, j+1, \dots, l-1$:
\begin{enumerate}[(i)]
\item $\angle(V_{n},\tilde V_{n})<2\delta$; \label{it:ind1}
\item $\perp(V_n,\bottomspacemat k{B_n})>\delta$ and
$\perp(\tilde V_n,\bottomspacemat k{B_n})>\delta$.
\label{it:ind2}
\end{enumerate}

Item \eqref{it:ind1} and the first part of \eqref{it:ind2} hold immediately
for the case $n=j$.
The second part of \eqref{it:ind2} holds because $\tilde V_j=V_j=
\bottomspacemat k{B_j}^\perp$ and
$\angle(\bottomspacemat k{B_j},\bottomspacemat k{\tilde B_j})<\delta$
 by Lemma \ref{lem:ly}\eqref{it:spacecont}.

Given that \eqref{it:ind1} and \eqref{it:ind2} hold for $n=m$, $B_m$ is
a good block and $\tilde{B}_m$ is a good perturbation,
Lemma~\ref{lem:ly}\eqref{it:contr}
implies that $\angle(V_{m+1},\topspacemat k{B_{m}})<\delta$,
$\angle(\tilde V_{m+1},\topspacemat k{B_{m}})<\delta$
so that $\angle(\tilde V_{m+1},V_{m+1})<2\delta$, yielding
\eqref{it:ind1} for $n=m+1$.

Making use of the induction hypothesis and Lemma~\ref{lem:goodBlocks}, we have that
$\angle(\topspaceDSunpert k{\sigma^{(m+1)N}\omega},
\topspacemat k{B_{m}})<\delta, \angle(\bottomspaceDSunpert k{\sigma^{(m+1)N}\omega},
\bottomspacemat k{B_{m}})<\delta$ and $\perp(\topspaceDSunpert k{\sigma^{(m+1)N}\omega},
\bottomspaceDSunpert k{\sigma^{(m+1)N}\omega)}>10\delta$.
Thus, we obtain \eqref{it:ind2} for $n=m+1$.

Hence using Lemma \ref{lem:ly}\eqref{it:det},
we see that $\log|\text{det}(\tilde B_n|_{\tilde V_n})|\ge
k\log \delta+\Xi_k(\tilde B_n)\ge k\log\delta-k\delta+\Xi_k(B_n)\ge\Xi_k(B_n)+2k\log\delta$,
where we made use of Lemma \ref{lem:ly}\eqref{it:pertsvs} for the second inequality.

Since $\Xi_k(\tilde B_{l-1}\cdots \tilde B_j)\ge
\sum_{i=j}^{l-1}\log|\det(\tilde B_i|\tilde V_i)|$,
summing yields
\begin{equation}\label{eq:pertLB}
\begin{split}
\Xi_k(\tilde B_{l-1}\cdots \tilde B_j)
&\ge 2(l-j)k\log\delta+\sum_{i=j}^{l-1}\Xi_k(B_i)\\
&\ge 2(l-j)k\log\delta+\Xi_k(B_{l-1}\cdots B_j),
\end{split}
\end{equation}
as required.
\end{proof}

\begin{lem}\label{lem:iteratedtopspacesclose}
Let the Hilbert-Schmidt cocycle, $A\colon\Omega\to\HS$ and all
parameters and perturbations
be as above. If $\bar\sigma^{iN}\omega\in\bar G$ for each $0\le i<n$,
then $\angle\left(
\bottomspacemat k {\cocyclepss{nN}\oom},
\bottomspacemat k {\cocycle{N}{\om}}\right)<\delta$.
\end{lem}

\begin{proof}
By the first part of \eqref{eq:pertLB}, $\Xi_k(\cocyclepss{nN}{\oom})>
\sum_{i=0}^{n-1}\Xi_k(\cocycle N{\sigma^{iN}\omega})+2nk\log\delta$.
Also, $\Xi_{k+1}(\cocyclepss{nN}\oom)\le
\sum_{i=0}^{n-1}\Xi_{k+1}(\cocyclepss{N}{\bar\sigma^{iN}\oom})
\le \sum_{i=0}^{n-1}\Xi_k(\cocyclepss{N}{\bar\sigma^{iN}\oom})+n\log 2
\le \sum_{i=0}^{n-1}\Xi_k(\cocycle N{\sigma^{iN}\om})+n\log 2+nk\delta
\le \sum_{i=0}^{n-1}\Xi_k(\cocycle N{\sigma^{iN}\om})-2nk\log\delta$
by Lemma \ref{lem:ly}\eqref{it:pertsvs}.
Since we have $\log s_{k+1}(\cocyclepss{nN}\oom)=\Xi_{k+1}(\cocyclepss{nN}\oom)
-\Xi_k(\cocyclepss{nN}\oom)$, we deduce $s_{k+1}(\cocyclepss{nN}\oom)
\le \delta^{-4nk}$.

On the other hand, if $v\in S$ is such that ${\perp}(v,\bottomspacemat k{
\cocycle N\om})>\delta$, an inductive argument exactly like the proof of Proposition
\ref{prop:GlueingGoodBlocks} shows that
$\|\cocyclepss{nN}\oom v\|
\ge (\frac\delta 3)^ne^{-n\delta}\prod_{i=0}^{n-1}s_k(\cocycle{nN}\om)
\ge (\delta^3K(\delta))^n$. The choice of $K(\delta)$ in Lemma
\ref{lem:ly} ensures $\|\cocyclepss{nN}\oom v\|>
s_{k+1}(\cocyclepss{nN}\oom)$, so that $v\not\in\bottomspacemat k
{\cocyclepss{nN}\oom}$.
\end{proof}

\begin{prop}\label{prop:almostXi}
Let $\omega$ be such that $\bar\sigma^{iN}\oom\in \bar G$ for $0\le i<n$.
Then for any $V$ such that ${\perp}(V,\bottomspacemat k{\cocycle N\om})>2\delta$,
one has $\log|\det(\cocyclepss{nN}\oom|_V)|\ge \Xi_k(\cocyclepss{nN}\oom)
+k\log\delta$.
\end{prop}

\begin{proof}
We argue as in Lemma \ref{lem:ly}\eqref{it:det}.
\begin{align*}
\log|\det(\cocyclepss{nN}\oom|_V)|&\ge
\Xi_k(\Pi_{\topspacemat k{\cocyclepss{nN}\oom}} \cocyclepss{nN}\oom\Pi_V)\\
&=\Xi_k(\cocyclepss{nN}\oom\Pi_{\bottomspacemat k{\cocyclepss{nN}\oom}^\perp}
\Pi_V)\\
&=\Xi_k(\cocyclepss{nN}\oom\Pi_{\bottomspacemat k{\cocyclepss{nN}\oom}^\perp})
+\Xi_k(\Pi_{\bottomspacemat k{\cocyclepss{nN}\oom}^\perp}
\Pi_V) \\
&\ge \Xi_k(\cocyclepss{nN}\oom)+
k\log{\perp}(\bottomspacemat k{\cocyclepss{nN}\oom},V),
\end{align*}
where we used Lemma \ref{lem:perp} for the last line. Lemma
\ref{lem:iteratedtopspacesclose}
and the triangle inequality  allow us to conclude.
\end{proof}

\section{Comparing perturbed and unperturbed bad blocks (Type I)}
\label{sec:badPertI}

We distinguish two ways in which a block can be bad: types I and II.
A \emph{type I bad block} is one where the unperturbed cocycle has
bad properties. On the other hand, a \emph{type II bad block} is one
where the unperturbed cocycle is well-behaved, but the perturbations
are wild.

Conditional on being in a type I bad block, the perturbations are
unconstrained, whereas conditional on being in a type II bad block
at least one perturbation is constrained to be large. For later use with
the type II bad blocks, we state some of the lemmas when one is
conditioned to be in a high probability event (but the high
probability event will be taken to be the whole space when dealing
with type I blocks.)

\begin{lem}\label{lem:randomnessdoesnthurt}
Let $k>0$. There exists a $C>0$ with the following property.
Let $T$ be a multi-variate normal Hilbert-Schmidt-valued random operator
whose entries have mean 0,
let $A\in\HS$ and let
$\Pi$ and $\Pi'$ be orthogonal projections onto $k$-dimensional subspaces of $H$.
Then for any subset $Q$ of $\HS$ such that
$\PP(T\in Q)\ge \frac 12$,
one has
$$
\E \Big( \big(\Xi_k(\Pi(A+T)\Pi') - \Xi_k(\Pi A\Pi')\big)^-
\Big| T\in Q\Big)
\ge -C,
$$
where $x^-$ denotes $\min(x,0)$. 
\end{lem}

\begin{proof}
We assume $\Xi_k(\Pi A\Pi')>-\infty$ as otherwise the result is trivial.

Let $\tilde\Pi$ be $\Pi$ composed with an isometry from the range of
$\Pi$ to $\R^k$ and similarly let $\tilde\Pi'$ be 
an isometry from $\R^k$ to the range of $\Pi'$.
Then we have $\Xi_k(\Pi B \Pi')=\log|\det(\tilde\Pi B\tilde \Pi')|$
for any bounded operator $B$ on $H$
so that we need to show
$$
\E \Big( \big(\log|\det(\tilde \Pi(A+T)\tilde\Pi')|
- \log|\det(\tilde \Pi A\tilde \Pi')|\big)^-\; \Big | T\in Q\Big)
\ge -C.
$$
Let $Y=\tilde\Pi A\tilde\Pi'$ and $Z=\tilde\Pi T\tilde\Pi'$, 
so that $Y$ is a fixed $k\times k$ matrix
and $Z$ is a $k\times k$ matrix-valued random variable with multivariate normal entries.
By our earlier assumption, $Y$ is
invertible, so let $X=ZY^{-1}$ (this also has multi-variate normal entries
for unconditioned $T$).
We then need a lower
bound for $\E\big(\log\det(I+X)\big| T\in Q\big)$.

The unconstrained matrix-valued random variable $X$ can be written as
$\sum_{l=1}^d N_l B^l$, where the $B^l$ are
fixed $k\times k$ matrices, $d$ is the
dimension of the support of $X$ (at most $k^2$
depending on the pattern of entries in the unperturbed $A$'s)
and the $N_l$ are independent standard normal random variables
(see for example \cite[Example 3.9.2]{Durrett}).

Let $\Psi$ denote the map from $\R^d$ to $M_{k\times k}$
defined by $x\mapsto \sum x_lB^l$.
Let $\mathcal S$ be the image under $\Psi$ of the unit sphere
and $\mu$ be the measure on $\mathcal S$ that is the push-forward
of the normalized volume measure on the unit sphere. The
unconditioned measure on $X$ is then the push forward of
$\mu\times C_dr^{d-1}e^{-r^2/2}\,dr$, where $C_d$ is chosen so that
$C_d\int r^{d-1}e^{-r^2/2}\,dr=1$. The conditioned measure on
$X$ (since the event being conditioned upon is of measure at least
$\frac 12$) is of the same form, but the density is multiplied by
a varying factor in the range [0,2].

It then suffices to lower bound
\begin{equation*}
2C_d
\int_\mathcal{S} d\mu(M)\int_0^\infty \log^-|\det(I+rM)|\, r^{d-1}e^{-r^2/2}\,dr.
\end{equation*}

In particular, it is enough to give a uniform lower bound for
$$
G(d,M)=\int_0^\infty \log^-|\det(I+rM)| \,r^{d-1}e^{-r^2/2}\,dr
$$
as $d$ ranges over the range 1 to $k^2$ and $M$ ranges over $M_{k\times k}$.

For each fixed $M$, write $p_M(r)=\det(I+rM)$, so that $p_M$ is a polynomial
of degree $k$ satisfying $p_M(0)=1$. Hence $p_M(r)$ can be written as a product
$\prod_{i=1}^{k}(1-b_ir)$. Define
\begin{align*}
F(d,b)&=\int_0^\infty \log^-|1-br|\,r^{d-1}e^{-r^2/2}\,dr,
\end{align*}
so that $G(d,M)\ge\sum_{i=1}^k F(d,b_i)$.  Hence it suffices to show that
$F(d,b)$ is uniformly bounded below as $b$ runs over the complex plane
and as $d$ runs over the range $1$ to $k^2$.

Next, notice that
$\log|1-br|\ge \log|1-\real(b)r|$, so $F(d,b)\ge F(d,|b|)$ and
it suffices to give a lower bound for positive real values of $b$.
Also
\begin{align*}
F(d,b)&=\frac{1}{b^d}\int_0^\infty \log^-|1-r|\,r^{d-1}e^{-r^2/(2b^2)}\,dr\\
&= \frac{1}{b^d}\int_0^2 \log|1-r|\,r^{d-1}e^{-r^2/(2b^2)}\,dr.
\end{align*}

For $b\ge \frac12$, $F(b,d)\ge \frac{1}{b^d}\int_0^2\log|1-r|r^{d-1}\,dr
\ge -2^d/b^d\ge -4^d$. For $0<b<\frac12$, one has
\begin{align*}
&F(d,b)\\
&\ge \frac 1{b^d}\int_0^{b^{d/(1+d)}}\log|1-r|r^{d-1}\,dr
+\frac1{b^d}\int_{b^{d/(1+d)}}^2\log|1-r|r^{d-1}e^{-r^2/(2b^2)}\,dr\\
&\ge -2\frac{1}{b^d}\int_0^{b^{d/(1+d)}}r^d\,dr
+(2/b)^d\exp(-1/(2b^\alpha))\int_0^2\log|1-r|\,dr\\
&\ge -2/(d+1)-2^{d+1}\exp(-1/(2b^\alpha))/b^d
\end{align*}
where $\alpha=2/(1+d)$.
This converges to $-2/(d+1)$ as $b$ approaches 0 from the right.
By continuity and compactness, for each of the finitely many values of $d$,
$F(d,b)$ is bounded below as $b$ ranges over $(0,\frac 12]$.
\end{proof}

\begin{prop}\label{prop:step2tilde}
Let $k>0$. Then there exists a $C_{\ref{prop:step2tilde}}$
with the following property.
For every finite sequence $A_0,\ldots,A_{n-1}$
of Hilbert-Schmidt operators, let $\Delta_0,\ldots,
\Delta_{n-1}$ be independent copies of the perturbation $\Delta$ as described above.
Let $\pertA_i$ denote $A_i+\epsilon\Delta_i$.

Then one has
$$
\E_{\Delta_0,\ldots,\Delta_{n-1}}\Big(\tilde\Xi_k(A_{n-1}^\ep\cdots A_0^\ep)-
\tilde\Xi_k(A_{n-1}\cdots A_0)
\Big)^-
\ge
-C_{\ref{prop:step2tilde}}n.
$$
\end{prop}

\begin{proof}
We have
\begin{align*}
&\E\Big(\tilde \Xi_k(A_{n-1}^\ep\cdots A_0^\ep)-\tilde \Xi_k(A_{n-1}\cdots A_0)
\Big)^-\\
&\ge \sum_{j=0}^{n-1} \E\Big(\tilde \Xi_k(A_{n-1}^\ep\cdots A_j^\ep A_{j-1}\cdots A_0)-
\tilde \Xi_k(A_{n-1}^\ep\cdots A_{j+1}^\ep A_{j}\cdots A_0)\Big)^-.
\end{align*}

We focus on giving a lower bound for one of the terms in the summation.
We write such a term as
$$
\E_{\Delta_j}\Big(\tilde \Xi_k (L(A_j+\ep\Delta)R) -
\tilde\Xi_k(LA_jR)\Big)^-.
$$

This expectation should be interpreted as being conditioned
on the values of $\Delta_{j+1},\ldots,\Delta_n$, so that
$L=(A_n+\ep\Delta_n)\cdots (A_{j+1}+\ep\Delta_{j+1})$.

The above expectation can be rewritten as:
\begin{equation}\label{eq:telescopesummand}
\E_{\Delta,\Delta'}\E_{\Delta_j}
\big[
\Xi_k(\Pi_k\Delta L(A+\ep\Delta_j)R\Delta'\Pi_k)-
\Xi_k(\Pi_k\Delta LAR\Delta'\Pi_k)\big]^-.
\end{equation}

Once $\Delta$ and $\Delta'$ are fixed, the inner expectation is
\begin{equation}\label{eq:fixD1D2}
\E_{\Delta_j}
\big[
\Xi_k(\Pi_k\Delta L(A+\ep\Delta_j)R\Delta'\Pi_k)-
\Xi_k(\Pi_k\Delta LAR\Delta'\Pi_k)\big]^-.
\end{equation}

Now let $\Pi$ be the orthogonal projection onto the orthogonal
complement of the kernel of $\Pi_k\Delta L$
and $\Pi'$ be the orthogonal projection onto the range of $R\Delta'\Pi_k$.
Then we have
\begin{align*}
\Xi_k(\Pi_k\Delta L(A+\ep\Delta_j)R\Delta'\Pi_k)&=
\Xi_k(\Pi_k\Delta L)+\Xi_k(\Pi(A+\ep\Delta_j)\Pi')+\Xi_k(R\Delta'\Pi_k);\\
\Xi_k(\Pi_k\Delta LAR\Delta'\Pi_k)&=
\Xi_k(\Pi_k\Delta L)+\Xi_k(\Pi A\Pi')+\Xi_k(R\Delta'\Pi_k);
\end{align*}

Now the quantity in \eqref{eq:fixD1D2} is
$$
\E_{\Delta_j}
\big[\Xi_k(\Pi(A+\ep\Delta_j)\Pi')-\Xi_k(\Pi A\Pi')\big]^-
$$
Applying Lemma \ref{lem:randomnessdoesnthurt} with $Q=\HS$, this is bounded below by $-C$,
independently of $\Delta$ and $\Delta'$, so that the
quantity in \eqref{eq:telescopesummand} is also bounded below by $-C$.
Since there are $n$ such terms, the statement in the lemma follows.
\end{proof}

\section{Type II bad block perturbations}\label{sec:badPertII}

Here we give an argument for good blocks in the base that have large
perturbations. We will obtain a drop in $\tilde\Xi_k$ over a bad
block of size $O(\log\epsilon)$ at worst, that is a drop of size $O(1)$ per symbol
since blocks are of length proportional to $|\log\ep|$.
However since the frequency of these
blocks is $O(e^{-C/\epsilon})$, the contribution of this drop to
the singular values of a large string of blocks is minuscule.

\begin{lem}\label{lem:integral}
There exists a constant $C>0$ such that if
$N$ is a standard normal random variable and $\Lambda>2$,
then for each $a\in\C$,
$$
\E\big(\log^-|1-aN|\big|N\ge \Lambda\big)\ge -C\log \Lambda.
$$
\end{lem}

Before giving the proof, let us give a heuristic explanation for why this should be true.
Conditional on $N\ge \Lambda$, the distribution of $N$ is approximately $\Lambda+
\text{Exp}(\Lambda)$, that is it typically takes values that are $\Lambda+O(1/\Lambda)$. The
worst case for the inequality
is approximately when $a=1/\Lambda$ and then the quantity inside the logarithm
is roughly $O(1/\Lambda^2)$.

\begin{proof}
We first recall that $\int_0^a\log x\,dx= a(\log a-1)$, so that the average value of
the logarithm function over $[0,a]$ is $\log a-1$.
We claim that for any interval $J$, one has
\begin{equation}\label{eq:logint}
\frac{1}{|J|}\int_J\log^-|x|\,dx\ge 2(\max_{x\in J}\log^-|x|-1).
\end{equation}
Indeed, this follows already for intervals $[0,a]$ with $0<a<1$,
and hence for sub-intervals of $[0,1]$ and $[-1,0]$. 
For intervals $[-a,b]$ with $a<0<|a|\le b\le 1$, we have
$1/(a+b)\int_{-a}^b\log^-|x|\,dx \ge 1/b\int_{-b}^b\log^-|x|\,dx=2(\log b-1)$.
If the interval $J$ is entirely outside $[-1,1]$, the inequality is trivial; and
if $J$ intersects $[-1,1]$, we have already established the inequality for $J\cap [-1,1]$,
from which the inequality for $J$ follows.

For $a\in\C$, the integrand in the statement
reduced if $a$ is replaced by $|a|$ so we
may assume $a>0$. If $a>2/\Lambda$, the integral is 0.

If $1/(3\Lambda)\le a\le 2/\Lambda$, let $I=[\Lambda,\frac 2a)$, the sub-interval
of $[\Lambda,\infty)$ where $\log|1-ax|<0$; and $J=[\frac 1a-\frac 1{a\Lambda^2},
\frac 1a+\frac1{a\Lambda^2}]$, the interval where $\log|1-ax|<-2\log \Lambda$.

The quantity to be bounded is
\begin{align*}
&\frac{\int_I\log^-|1-ax|e^{-x^2/2}\,dx}
{\int_\Lambda^\infty e^{-x^2/2}\,dx}
\ge
\frac{\int_I\log^-|1-ax|e^{-x^2/2}\,dx}
{\int_I e^{-x^2/2}\,dx}\\
=\, &
\frac{\int_{I\cap J}\log^-|1-ax|e^{-x^2/2}\,dx + \int_{I\setminus J}
\log^-|1-ax|e^{-x^2/2}\,dx}
{\int_{I\cap J}e^{-x^2/2}\,dx + \int_{I\setminus J}
e^{-x^2/2}\,dx}
\end{align*}

The ratio of the two integrals over $I\setminus J$ is bounded below
by $-2\log \Lambda$.
Using \eqref{eq:logint}, the ratio of the two integrals over $I\cap J$ is bounded
below by $2(-2\log \Lambda-1)\max_{I\cap J}e^{-x^2/2}/\min_{I\cap J}e^{-x^2/2}
\ge 2e^{2/(a^2\Lambda^2)}(-2\log \Lambda-1)\ge -2e^{18}(2\log \Lambda+1)$.
Since both ratios are bounded
below by a constant multiple of $\log \Lambda$, so is the ratio of the sums.

If $a<1/(3\Lambda)$, we argue similarly. In this case, we let
$J=[\frac{1}{2a},\frac{3}{2a}]$. On $I\setminus J$, $\log|1-ax|$ is bounded
below by $-\log 2$, so that
$$
\frac{\int_{I\setminus J}\log|1-ax|e^{-x^2/2}\,dx}
{\int_\Lambda^\infty e^{-x^2/2}\,dx}\ge
\frac{\int_{I\setminus J}\log|1-ax|e^{-x^2/2}\,dx}
{\int_{I\setminus J} e^{-x^2/2}\,dx}\ge-\log 2.
$$
On $I\cap J$, we have $e^{-x^2/2}\le e^{-1/(8a^2)}$. Also $\int_\Lambda^\infty
e^{-x^2/2}\,dx\ge e^{-\Lambda^2/2}/(2\Lambda)$, using \cite[Theorem 1.2.3]{Durrett}.
Hence
$$
\frac{\int_{I\cap J}\log^-|1-ax|e^{-x^2/2}\,dx}
{\int_\Lambda^\infty e^{-x^2/2}\,dx}\ge
\frac{2e^{-1/(8a^2)}(-\log 2-1)\frac1a}{e^{-\Lambda^2/2}/(2\Lambda)},
$$
using \eqref{eq:logint}.
When $a=1/(3\Lambda)$, this is  $4(-\log 2-1)3\Lambda^2e^{-5\Lambda^2/8}$
and the lower bound increases as $a$ is further reduced. Minimizing this expression
over $\Lambda$, we see that
there is a $C$, independent of $\Lambda$, such that
$\E\big(\log^-|1-aN|\big|N\ge \Lambda\big)\ge -C$
for all $|a|<1/(3\Lambda)$.
\end{proof}

\begin{lem}\label{lem:randomnessdoesnthurtbad}
Let $k>0$ and $\Delta$ be as throughout the article. There exists $C>0$ such that
for all sufficiently small $\epsilon>0$,
for each $a,b$ and each pair of $k$-dimensional orthogonal projections
$\Pi$ and $\Pi'$,
$$
\E\Big(\big(\Xi_k(\Pi(A+\epsilon\Delta)\Pi')-\Xi_k(\Pi A \Pi')\big)^-
\Big|\textsf{Wild}_{a,b}\Big)
> C(\log\epsilon-a-b),
$$
where $\textsf{Wild}_{a,b}$ is the event that $\Delta$ satisfies
$|\Delta_{l,m}|<(\frac23)^{l+m}\epsilon^{-1/2}$
for each $(l,m)$ that is lexicographically smaller than $(a,b)$ and
$|\Delta_{a,b}|\ge \epsilon^{-1/2}(\frac23)^{a+b}$ (where $(l,m)$ is
lexicographically smaller than $(a,b)$ if $l<a$ or $l=a$ and $m<b$).
\end{lem}

\begin{proof}
We deal with the case $\Delta_{a,b}$ positive. The case where it is
negative is exactly analogous.
Let $B_{a,b}$ be the collection of those $\Delta$ satisfying $\Delta_{a,b}
\ge \epsilon^{-1/2}(\frac23)^{a+b}$ (and no other condition).
The argument of Lemma \ref{lem:goodPert} shows that
$\PP(\textsf{Wild}_{a,b}|B_{a,b})>\frac 12$.
This allows us to deduce as in the proof of Lemma
\ref{lem:randomnessdoesnthurt} that
\begin{align*}
&\E\Big(\big(\Xi_k(\Pi(A+\epsilon\Delta)\Pi')-\Xi_k(\Pi A \Pi')\big)^-\Big|
\Delta\in \textsf{Wild}_{a,b}\Big)\\
&>2
\E\Big(\big(\Xi_k(\Pi(A+\epsilon\Delta)\Pi')-\Xi_k(\Pi A \Pi')\big)^-\Big|
\Delta\in B_{a,b}\Big)
\end{align*}
Hence it suffices to show that
$$
\E\Big(\big(\Xi_k(\Pi(A+\epsilon\Delta)\Pi')-\Xi_k(\Pi A \Pi')\big)^-\Big|
\Delta\in B_{a,b}\Big)> C(\log\epsilon-a-b).
$$

Using the same reduction as in Lemma \ref{lem:randomnessdoesnthurt},
the calculation reduces to
showing that there is a $C$ such that for sufficiently small $\epsilon>0$,
one has for an arbitrary $k\times k$ multi-variate normal matrix-valued random
variable, $R$, whose entries have zero mean and for an arbitrary rank 1 $k\times k$ matrix $Y$,
$$
\E_{N,R}\Big(\Xi_k(I+R+\epsilon NY)^-\big | N>2^{a+b}
\epsilon^{-1/2}\Big)\ge C(\log\epsilon-a-b),
$$
where $N$ is an independent standard normal random variable.
First fixing $N$ and taking the expectation over $R$ using
Lemma \ref{lem:randomnessdoesnthurt} (taking $Q$ to be the full range of $\Delta$),
 we obtain
\begin{align*}
&\E_{N,R}\Big(\Xi_k(I+R+\epsilon NY)^-\big | N>2^{a+b}
\epsilon^{-1/2}\Big)\\
&\ge
\E_N\Big(\Xi_k(I+\epsilon NY)^-\big|N>2^{a+b}\epsilon^{-1/2}\Big)-C.
\end{align*}
Hence it suffices to show
$$
\E\big(\Xi_k(I+\epsilon NY)^-\big|N>2^{a+b}\epsilon^{-1/2}
\big)\ge C(\log\epsilon-a-b).
$$
Since $Y$ has rank 1, the polynomial $\det(I+tY)$ is of the form $1+at$.
To see this, notice the determinant is unchanged if $I+tY$ is conjugated
by an orthogonal matrix, $O$. Then choose $O$ so that the first column
spans the range of $Y$ so that $O^{-1}(I+tY)O=I+t\tilde Y$, where
$\tilde Y$ has only one non-zero row. $\det(I+tY)$ is then $1+t\tilde Y_{1,1}$.
Hence we are seeking a lower bound for
$$
\E(\log^-|1+cN|\big|N>2^{a+b}\epsilon^{-1/2}),
$$
which is of the desired form by Lemma \ref{lem:integral}.
\end{proof}

\begin{prop}\label{prop:badtriangleineq}
There exists a $C_{\ref{prop:badtriangleineq}}>0$ with the following property.
For any $m>0$, let $B$ be the event that at least one of the perturbations
$\Delta_0,\ldots,\Delta_{m-1}$
is wild. Then
\begin{equation*}
\E\big(\tilde\Xi_k(\cocyclepss m\oom)\big|B\big)\ge
\tilde\Xi_k(\cocycle m\om)+C_{\ref{prop:badtriangleineq}}(\log\epsilon-m).
\end{equation*}
\end{prop}

\begin{proof}
We write $B$ as $B_0\cup\ldots\cup B_{m-1}$, where $B_i$
is the event that the $i$th perturbation matrix is wild, and all previous ones are tame.
Since the $B_i$ are disjoint,
it suffices to establish that there is a $C>0$ such that for each $i$,
\begin{equation}\label{eq:badtri2}
\E(\tilde\Xi_k(\cocyclepss m\oom)|B_i)\ge
\tilde\Xi_k(\cocycle m\om)+C(\log\epsilon-m).
\end{equation}

We argue as in Proposition \ref{prop:step2tilde}:
\begin{align*}
&\E\Big(\tilde\Xi_k(\cocyclepss m\oom)-\tilde\Xi_k(\cocycle m\om)\Big|B_i\Big)\\
&=\sum_{j=0}^{m-1}
\E\Big(\tilde\Xi_k\big(\cocyclepss {m-j}{\bar\sigma^{j}\oom}
\cocycle j\om\big)-
\tilde\Xi_k\big(\cocyclepss{m-j-1}{\bar\sigma^{j+1}\oom}
\cocycle{j+1}\om\big)\Big|B_i\Big)
\end{align*}
As in Proposition \ref{prop:step2tilde}, finding lower bounds for this reduces
to finding lower bounds for $\E\Big(
\tilde\Xi_k(\Pi \pertA_{\bar\sig^j\oom}\Pi')-
\tilde\Xi_k(\Pi A_{\sig^j\om}\Pi')\Big| B_i\Big)$.

In this case, for $j>i$, the conditional distribution of $\Delta_j$ is
the same as the distribution used in Lemma \ref{lem:randomnessdoesnthurt}
with $Q=\HS$, so that lemma gives a bound
\begin{equation}\label{eq:constLB}
\E\Big(\tilde\Xi_k\big(\cocyclepss{n-j}{\bar\sigma^{j}\oom}
\cocycle j\om)-
\tilde\Xi_k\big(\cocyclepss{n-j-1}{\bar\sigma^{j+1}\oom}
\cocycle{j+1}\om\big)\Big|B_i\Big)\ge-C.
\end{equation}

In the case $j<i$, $\Delta_j$ is conditioned to be tame.
By Lemma \ref{lem:goodPert}, this is a set of probability (much)
greater than $\frac 12$, so that Lemma \ref{lem:randomnessdoesnthurt}
gives a similar bound to \eqref{eq:constLB}.

Finally, we address the term with $j=i$. Given that $\Delta_i$ is wild, the probability
that the first oversized entry occurs in the $(a,b)$ coordinate is
$O(\exp(-\frac12\epsilon^{-1}(2^{2a+2b}-1)))$ (as seen from the estimate
$\PP(N>t)\approx (2\pi)^{-1/2}e^{-t^2/2}/t$ for large $t$
\cite[Theorem 1.2.3]{Durrett}).

Hence by conditioning and using Lemma \ref{lem:randomnessdoesnthurtbad}, we obtain
\begin{equation}\label{eq:badLB}
\E\Big(\tilde\Xi_k\big(\cocyclepss{m-i}{\bar\sig^i\oom}\cocycle i\om)-
\tilde\Xi_k\big(\cocyclepss{m-i-1}{\bar\sigma^{i+1}\oom}
\cocycle{i+1}\om\big)\Big|B_i\Big)>C(\log\epsilon-1).
\end{equation}
Combining equations \eqref{eq:constLB} and the equation \eqref{eq:badLB}, we obtain the
statement of the proposition.
\end{proof}

\section{Joining good and bad blocks}\label{sec:splitting}

\begin{lem}\label{lem:centred}
For all $k\in\N$, there is a constant $C>0$
 such that for any $A\in\HS$,
any orthogonal projections
$\Pi_1$ and $\Pi_2$ onto $k$-dimensional subspaces, and any
$Q\subset \HS$ such that $\PP(\Delta\in Q)\ge \frac 12$,
one has

$$
\E\Xi_k\big(\Pi_1(A+\Delta)\Pi_2\big|\Delta\in Q\big)
\ge \E\Xi_k\big(\Pi_1\Delta\Pi_2\big|\Delta\in Q\big)-C.
$$
\end{lem}

\begin{proof}
Let $\tilde \Pi_1$ be
an isometry from the range of $\Pi_1$ to $\R^k$. Similarly
let $\tilde \Pi_2$ be the post-composition of $\Pi_2$ with an
isometry from $\R^k$ to the span of the range
of $\Pi_2$. Let $\tilde A=\tilde \Pi_1A\tilde \Pi_2$
and let $\tilde\Delta=\tilde \Pi_1\Delta\tilde \Pi_2$ be
the $k\times k$ multi-variate normal induced from the unconditioned
distribution of $\Delta$.

As in Lemma \ref{lem:randomnessdoesnthurt},
we radially disintegrate the random variables $\tilde\Delta$,
writing $\tilde\Delta$ as $t\tilde M$, where $\tilde M$
belongs to a `unit sphere' equipped with a normalized probability
measure and $t$ having an absolutely continuous
distribution on $[0,\infty)$ with density $r^{k^2-1}e^{-r^2/2}/\Gamma(k^2/2)$.
On conditioning on $\Delta\in Q$, the density is bounded above by
$2r^{k^2-1}e^{-r^2/2}/\Gamma(k^2/2)$
We prove that there is a $C>0$ such that for all $\tilde M$ of rank $k$,
$$
\frac{2}{\Gamma(k^2/2)}\int_0^\infty \Big(\Xi_k(\tilde A+r\tilde M)-
\Xi_k(r\tilde M)\Big)^-r^{k^2-1}
e^{-r^2/2}\,dr > -C.
$$

Notice that since the matrices are $k\times k$, $\Xi_k$ is just
the logarithm of the absolute value of the determinant. Let
$p(r)=\det(\tilde A+r\tilde M)/\det(r\tilde M)$, a polynomial
in powers of $1/r$ of degree at most $k$ with constant coefficient 1.
It can therefore be expressed as
$p(r)=\prod_{i=1}^d (1-b_i/r)$, with $d\le k$.

We are trying to bound
$$
\int_0^\infty \log^-|p(r)|r^{k^2-1}e^{-r^2/2}\,dr\ge
\sum_{i=1}^k \int_0^\infty\log^- |1-b_i/r|r^{k^2-1}e^{-r^2/2}\,dr.
$$

As in the proof of Lemma \ref{lem:integral}, it suffices to give a bound in the case where $b>0$.
We have
$$
\int_0^\infty\log^- |1-b/r|r^{k^2-1}e^{-r^2/2}\,dr
=
\int_{b/2}^\infty \log^-|1-b/r|r^{k^2-1}e^{-r^2/2}\,dr.
$$
The logarithm is bounded below by $-\log 2$ on
$(2b,\infty)$, so that the contribution from this range is
at least $-\Gamma(k^2/2)\log 2$.
For the contribution from the range $[\frac b2,2b]$,
we have a lower bound of $-16(2b)^{k^2-2}e^{-b^2/8}$ (obtained
by bounding $e^{-r^2/2}$ above by $e^{-b^2/8}$).
Hence we obtain the required uniform lower bound.
\end{proof}

The following lemma plays a key role, as it provides an approximate
super-additivity property for $\tilde\Xi_k$ (making strong use of the
nature of the perturbations), complementing the well-known
sub-additivity property of $\Xi_k$.

\begin{lem}\label{lem:join}
There exists $C>0$ such that if $\Delta$ is distributed as above and $Q$ is
any subset of $\HS$ such that
$\PP(Q\in\Delta)\ge \frac 12)$, then
$$
\E\big(\tilde\Xi_k(L(A+\ep\Delta)R)\big|\Delta\in Q\big)
\ge \tilde\Xi_k(L)+\tilde\Xi_k(R)-k|\log\ep|
-C.
$$
\end{lem}

\begin{proof}
We may assume that $L$ and $R$ have rank at least $k$ as otherwise
there is nothing to
prove. Recalling the definition of $\tilde\Xi$, we have
\begin{align*}
&\E\big(\tilde\Xi_k(L(A+\ep\Delta)R)\big|\Delta\in Q\big)\\
&=\E_{\Delta_1,\Delta_2}\E\big(\Xi_k(\Pi_k\Delta_1 L(A+\ep\Delta)R \Delta_2\Pi_k)
\big|\Delta\in Q\big)\text{ and}\\
&\E\big(\tilde\Xi_k( L(\ep\Delta) R)\big|\Delta\in Q\big)
=\E_{\Delta_1,\Delta_2}\E\big(\Xi_k( \Pi_k\Delta_1 L(\ep\Delta) R
\Delta_2\Pi_k)\big|\Delta\in Q\big)
\end{align*}
We first show that for fixed $\Delta_1$ and $\Delta_2$,
\begin{equation}\label{eq:lemjoin-toprove}
\begin{split}
&\E\big(\Xi_k(\Pi_k\Delta_1 L(A+\ep\Delta)R \Delta_2\Pi_k)
\big|\Delta\in Q\big)\\
&\ge
\E\big(\Xi_k( \Pi_k\Delta_1 L(\ep\Delta) R \Delta_2\Pi_k)\big|\Delta\in Q\big)-C.
\end{split}
\end{equation}
We have
$
\Xi_k(\Pi_k\Delta_1 L(A+\ep\Delta)R \Delta_2\Pi_k)
=\Xi_k(\Pi_k\Delta_1 L)+\Xi_k(\overline{\Pi}(A+\ep\Delta)\overline{\overline{\Pi}})
+\Xi_k(R\Delta_2\Pi_k)$
and
$\Xi_k(\Pi_k\Delta_1 L(\ep\Delta)R \Delta_2\Pi_k)
=\Xi_k(\Pi_k\Delta_1 L)+\Xi_k(\overline \Pi(\ep\Delta)\overline{\overline{\Pi}})
+\Xi_k(R\Delta_2\Pi_k)$, where $\overline\Pi$ is the orthogonal
projection onto the $k$-dimensional
orthogonal complement of the kernel of $\Pi_k\Delta_1 L$ and
$\overline{\overline\Pi}$
is the orthogonal projection onto the range of $R\Delta_2\Pi_k$.
Hence
\begin{align*}
&\Xi_k(\Pi_k\Delta_1 L(A+\ep\Delta)R \Delta_2\Pi_k)-
\Xi_k(\Pi_k\Delta_1 L(\ep\Delta)R \Delta_2\Pi_k)\\
&=
\Xi_k(\overline{\Pi}(A+\ep\Delta)\overline{\overline{\Pi}})-
\Xi_k(\overline \Pi(\ep\Delta)\overline{\overline{\Pi}})\\
&=\Xi_k(\overline{\Pi}(\tfrac1\ep A+\Delta)\overline{\overline{\Pi}})-
\Xi_k(\overline \Pi\Delta\overline{\overline{\Pi}}).
\end{align*}
Taking an expectation as $\Delta$ runs over $Q$ and using Lemma \ref{lem:centred},
we obtain \eqref{eq:lemjoin-toprove}. Hence, taking the expectation over $\Delta_1$
and $\Delta_2$, we have
\begin{align*}
\E\big(\tilde\Xi_k(L(A+\ep\Delta)R)\big|\Delta\in Q\big)&\ge
\E\big(\tilde\Xi_k( L(\ep\Delta) R)\big|\Delta\in Q\big)-C\\
&=\E\big(\tilde\Xi_k( L\Delta R)\big|\Delta\in Q\big)-C+k\log\ep.
\end{align*}

For the last part of the argument,
we have
\begin{align*}
&\E\big(\tilde\Xi_k( L\Delta R)\big|\Delta\in Q\big)=
\E_{\Delta_1,\Delta_2}\E_\Delta\big(\Xi_k(\Pi_k\Delta_1L\Delta R\Delta_2\Pi_k)
\big|\Delta\in Q\big)\\
&=\E_{\Delta_1,\Delta_2}\Big(
\Xi_k(\Pi_k\Delta_1L\overline\Pi)+
\E_\Delta\big(\Xi_k(\overline\Pi\Delta\overline{\overline\Pi})\big|\Delta\in Q\big)
+\Xi_k(\overline{\overline\Pi}R\Delta_2\Pi_k)\Big),
\end{align*}
where $\overline\Pi$ and $\overline{\overline\Pi}$ are as above.
By Corollary \ref{cor:PS}, the middle term is $\E_{\Delta_3}
\Xi_k(\overline\Pi\Delta_3\Pi_k)+\E_{\Delta_4}
\Xi_k(\Pi_k\Delta_4\overline{\overline \Pi})\pm C$. Substituting
and recombining the expressions, we get
\begin{align*}
&\E\big(\tilde\Xi_k( L(A+\epsilon\Delta) R)\big|\Delta\in Q\big)\\
&\ge \E_{\Delta_1,\Delta_3}\Xi_k(\Pi_k\Delta_1L\Delta_3\Pi_k)+
\E_{\Delta_2,\Delta_4}\Xi_k(\Pi_k\Delta_4R\Delta_2\Pi_k)- C+k\log\epsilon\\
&=\tilde\Xi_k(L)+\tilde\Xi_k(R)-C+k\log\epsilon,
\end{align*}
as required.
\end{proof}

Since the statement includes the case where $\Delta$ is conditioned to
lie in a large set,
this is sufficient to cover the case where $\Delta$ is conditioned to be tame.
We need a version of this inequality to deal with the case where
$\Delta$ is constrained to be wild.

\begin{lem}\label{lem:Mahler}
There exists $C>0$ such that for all polynomials, $p(x)$,
one has
$$
\left|
\int_{-\infty}^\infty\frac{e^{-x^2/2}}{\sqrt{2\pi}}\log |p(x)|\,dx
-\log M(p)\right|
\le C\deg(p),
$$
where $M(p)$ is the Mahler measure of $p$: If $p(x)=a(x-z_1)(x-z_2)\cdots
(x-z_k)$, then $M(p)=a\prod_{|z_i|>1}|z_i|$.
\end{lem}

\begin{proof}
Write $p(x)$ as $a(x-z_1)\cdots (x-z_k)$. The inequality then follows from
$$
\left|\int_{-\infty}^\infty \frac{e^{-x^2/2}}{\sqrt{2\pi}}\log|x-z|\,dx
-\log^+|z|\right|\le C.
$$
While we will not give all the details, the idea is to notice that the integral can be
expressed as $\E\log|N-z|$ where $N$ is a standard normal random variable.
If $z$ is small, then this is the integral of a function with a logarithmic singularity.
If $z$ is large, then since $N$ is concentrated near 0, the integrand
is close to $\log|z|$ with very high probability.
\end{proof}

\begin{lem}\label{lem:Mahler2}
For each $k>0$, there exists a constant $C$ such that for each
polynomial $p(x)=\sum_{i=0}^k a_ix^i$, one has
$$
\big|\log M(p)-\max\log|a_i|\big|\le C.
$$
\end{lem}

The proof can be found in Lang's book \cite[Theorem 2.8]{Lang}.

\begin{lem}\label{lem:intlog}
Let $\Lambda>2$ and let $N$ be a standard normal random variable. There exists a $C>0$
such that for all $a,b\in\C$,
\begin{equation*}
\E\Big(\log|a+bN|\Big|N>\Lambda\Big)\ge
\max(\log|a|,\log|b|)-C\log\Lambda.
\end{equation*}
\end{lem}

\begin{proof}
The case where $|a|>|b|$ follows from Lemma \ref{lem:integral} (writing
$\log|a+bN|=\log|a|+\log|1+\frac baN|$). If $|b|\ge|a|$, then $|a+bN|\ge |b|\Lambda/2$
whenever $N>\Lambda$. The result follows.
\end{proof}

\begin{lem}\label{lem:joinbad}
There exists a constant $C>0$ such that for all $i,j$,
\begin{align*}
&\E\tilde\Xi_k(L(A+\ep\Delta)R\big|\mathsf{Wild}_{i,j})\\
\ge\, &\tilde\Xi_k(L)+\tilde\Xi_k(R)-C|\log\ep|-C(i+j+1),
\end{align*}
where $\mathsf{Wild}_{i,j}$ is the event that $|\Delta_{i,j}|\ge
(\frac 23)^{i+j}\epsilon^{-1/2}$
and $|\Delta_{a,b}|<(\frac 23)^{a+b}\epsilon^{-1/2}$ for all pairs
$(a,b)$ that are lexicographically
smaller than $(i,j)$.
\end{lem}

\begin{proof}
As in the proof of Lemma \ref{lem:join}, the proof reduces
to showing a version of Lemma \ref{lem:centred}:
$$
\E\Xi_k\big(\Pi_1(A+\epsilon\Delta)\Pi_2\big|\mathsf{Wild}_{i,j}\big)
\ge \E\Xi_k(\Pi_1\ep\Delta\Pi_2)-C(i+j+1).
$$
We first compare
$\E\Xi_k\big(\Pi_1(A+\epsilon\Delta)\Pi_2\big|\mathsf{Wild}_{i,j}\big)$
to $\E\Xi_k\big(\Pi_1(A+\epsilon\Delta)\Pi_2\big|\mathsf{Tame}_{i,j}\big)$,
where $\mathsf{Tame}_{i,j}$ is the event that
$|\Delta_{a,b}|<(\frac 23)^{a+b}\epsilon^{-1/2}$ for all pairs
$(a,b)$ that are lexicographically
smaller than $(i,j)$. Fixing all entries of $\Delta$ other than
$\Delta_{i,j}$, this amounts to
comparing $\E\big(\log|\det (B+NZ)|\big|N>2^{i+j}\epsilon^{-1/2}\big)$
to $\E\big(\log|\det (B+NZ)|\big)$, where $B$ is an invertible $k\times k$ matrix and
$Z$ is rank 1. As pointed out in Lemma \ref{lem:randomnessdoesnthurtbad},
$\det(B+NZ)=a+bN$ for constants
$a$ and $b$, so that
it suffices to compare $\E\big(\log|a+bN|\big|N>2^{i+j}\epsilon^{-1/2}\big)$
to $\E\log|a+bN|$. By Lemma \ref{lem:intlog}, the first of these is at least
$\max(\log|a|,\log|b|)-C(i+j+\log\epsilon)$ and by Lemmas \ref{lem:Mahler}
and \ref{lem:Mahler2}, the second of these is within $C$ of $\max(\log|a|,\log|b|)$.
We deduce that
$$
\E\Xi_k\big(\Pi_1(A+\epsilon\Delta)\Pi_2\big|\mathsf{Wild}_{i,j}\big)
>
\E\Xi_k\big(\Pi_1(A+\epsilon\Delta)\Pi_2\big|\mathsf{Tame}_{i,j}\big)
-C(i+j+\log\epsilon).
$$
Hence, using the same cancellation argument that occurs in Lemma \ref{lem:join}, we have
$$
\E\tilde\Xi_k(L(A+\ep\Delta)R|\textsf{Wild}_{i,j})
\ge
\E\tilde\Xi_k(L(A+\ep\Delta)R|\textsf{Tame}_{i,j})-C(i+j+\log\ep).
$$
Finally
using Lemma \ref{lem:join} to bound $\E\tilde\Xi_k(L(A+\ep\Delta)R|\textsf{Tame}_{i,j})$,
 the result follows.
\end{proof}

\begin{prop}
\label{prop:splitting}
There exists $C_{\ref{prop:splitting}}>0$ with the following property:
Let $L$, $R$, and $A$ be Hilbert-Schmidt operators and let $\Delta$ be the
multivariate normal perturbation described earlier. Then each of
$\E\tilde\Xi_k(L(A+\epsilon\Delta)R)$, $\E\big(\tilde\Xi_k(L(A+\epsilon\Delta)R)
\big|\Delta\text{ is wild}\big)$ and
$\E\big(\tilde\Xi_k(L(A+\epsilon\Delta)R)\big|\Delta\text{ is tame}\big)$
is bounded below by $\tilde\Xi_k(L)+\tilde\Xi_k(R)+
C_{\ref{prop:splitting}}\log\epsilon$.
\end{prop}

\begin{proof}
The cases of $\E\tilde\Xi_k(L(A+\epsilon\Delta)R)$,
$\E\big(\tilde\Xi_k(L(A+\epsilon\Delta)R)
\big|\Delta\text{ is tame}\big)$ are handled by Lemma \ref{lem:join}.
The case of $\E\big(\tilde\Xi_k(L(A+\epsilon\Delta)R)\big|\Delta\text{ is wild}\big)$
is handled using Lemma \ref{lem:joinbad} by conditioning on the first entry of $\Delta$
that is large analogously to the end of the proof of
Proposition \ref{prop:badtriangleineq}.
\end{proof}

\section{Comparison of $\Xi_k$ and $\tilde\Xi_k$}\label{sec:Comparison}

\begin{lem}\label{lem:xicomp}
Let $C_k$ be the expected value of $\log|\det N_k|$ where $N_k$ is a
$k\times k$ matrix-valued random variable with independent standard normal entries.
Let $n\ge k$, let $A$ be an $n\times n$ matrix and let $N$ be a $k\times n$
matrix-valued random variable with independent standard normal entries.
Then $\E\Xi_k(NA)\ge \Xi_k(A)+C_k$.
\end{lem}

\begin{proof}
Write $A=UDV$ where $U$ and $V$ are orthogonal and $D$ is
diagonal with decreasing entries.
Then by an argument like that in Lemma \ref{lem:PS} (computing covariances between elements)
$NU$ has the same distribution as $N$, so that we have
$\E\Xi_k(NA)=\E\Xi_k(NUDV)=\E\Xi_k(ND)\ge \E\Xi_k(ND\Pi_k)$.
Notice that since $D$ is diagonal,
$ND\Pi_k$ has the form $\begin{pmatrix}N_kD_k|0\end{pmatrix}$,
where $N_k$ is the  left $k\times k$ submatrix of $N$ and $D_k$ is the top left
$k\times k$ submatrix of $D$. Hence $\E\Xi_k(ND\Pi_k)=\E\Xi_k(N_kD_k)=C_k+\Xi_k(D_k)
=C_k+\Xi_k(A)$ as required.
\end{proof}

\begin{lem}\label{lem:trunc}
Let $A$, $B$ and $C$ be Hilbert-Schmidt matrices, and let $A_n=\Pi_nA\Pi_n$.
Then $\Xi_k(BA_nC)\to \Xi_k(BAC)$ as $n\to\infty$.
\end{lem}

\begin{proof}
Let $R_n=A-A_n$, so that $\|R_n\|\to 0$. We have $|s_i(BA_nC)-s_i(BAC)|\le
\|B\|\cdot\|R_n\|\cdot\|C\|$ for each
$i$ so that
$s_i(BA_nC)\to s_i(BAC)$ for each $i$. The conclusion follows.
\end{proof}

\begin{prop}\label{prop:XikvsXiktilde}
Let $k>0$. Then there exists a constant $C_{\ref{prop:XikvsXiktilde}}$
such that for an arbitrary Hilbert-Schmidt operator $A$ on $H$,
$$
\tilde\Xi_k(A)\ge \Xi_k(\diagpar3A\diagpar3)-C_{\ref{prop:XikvsXiktilde}}.
$$
\end{prop}

\begin{proof}
We have $\tilde\Xi_k(A)=\E_{\Delta,\Delta'}\Xi_k(\Pi_k\Delta A\Delta'\Pi_k)$ where
$\Delta$ and $\Delta'$ are independent copies of the perturbation operator.
Since $\Xi_k(\Pi_k\Delta A_n\Delta'\Pi_k)\le
k\log\|\Pi_k\Delta A_n\Delta'\Pi_k\|_\op
\le k\log(\|\Delta\|_\op\cdot\|A_n\|_\op\cdot\|\Delta'\|_\op)$;
$\|A_n\|_\op\le \|A\|_\HS$ and $\E\log\|\Delta\|_\op<
\E\|\Delta\|_\op\le \E\|\Delta\|_\HS<\infty$, we see that the family of functions,
$(\Delta,\Delta')\mapsto \Xi_k(\Pi_k\Delta A_n\Delta'\Pi_k)$ is dominated by an integrable
function. Hence, by the Reverse Fatou Lemma and Lemma \ref{lem:trunc}, we have
$$
\limsup_{n\to\infty}
\tilde\Xi_k(A_n)=\limsup_{n\to\infty}
\E\Xi_k(\Pi_k\Delta A_n\Delta'\Pi_k)\le \tilde\Xi_k(A).
$$
However, we have
\begin{align*}
\E_{\Delta,\Delta'}\Xi_k(\Pi_k\Delta A_n\Delta'\Pi_k)
=
\E_{\Delta,\Delta'}\Xi_k(\Delta_{k\times n}A_n\Delta'_{n\times k}),
\end{align*}
where $\Delta_{k\times n}$ denotes the random Hilbert Schmidt operator $\Delta$
with all entries outside the top left $k\times n$ corner replaced by 0's (and
$\Delta'_{n\times k}$ similarly).
Hence
\begin{align*}
\tilde\Xi_k(A_n)&=\E_{N,N'}\Xi_k\big((\diagpar 3)_{k\times k}N_{k\times n}
(\diagpar 3)_{n\times n}
A_n(\diagpar 3)_{n\times n}N'_{n\times k}(\diagpar 3)_{k\times k}\big)\\
&= \E_{N,N'}\Xi_k\big(N_{k\times n}(\diagpar 3)_{n\times n}
A_n(\diagpar 3)_{n\times n}N'_{n\times k}\big)-k(k-1)\log 3
\end{align*}

Applying Lemma \ref{lem:xicomp} twice, we deduce $\tilde\Xi_k(A_n)
\ge \Xi_k(\diagpar 3 A_n\diagpar 3)+C$, so that on taking the limit,
we deduce $\tilde\Xi_k(A)\ge \Xi_k(\diagpar 3A\diagpar 3)+C$ as required.
\end{proof}

\begin{cor}\label{cor:twoglue}
There is a $C_{\ref{cor:twoglue}}$
with the following property. Let $L$, $R$, $A$ and $A'$ be
Hilbert-Schmidt operators and  $\Delta$ and
$\Delta'$ be independent copies of the standard perturbation.
Then we have
$$
\E\tilde\Xi_k(L(A'+\epsilon\Delta')(A+\epsilon\Delta)R)
\ge \tilde\Xi_k(L)+\tilde\Xi_k(R)+C_{\ref{cor:twoglue}}\log\epsilon.
$$
The same inequality holds if either or both of $\Delta$ and $\Delta'$ are
constrained to be either tame or wild (or one of each).
\end{cor}

\begin{proof}
Let $L'=L(A'+\epsilon\Delta')=L(A'+\epsilon\Delta')I$.
By Proposition \ref{prop:splitting},
$\E_{\Delta'}\tilde\Xi_k(L')\ge \tilde\Xi_k(L)+\tilde\Xi_k(I)+C_{\ref{prop:splitting}}
\log\epsilon$, with this inequality
still satisfied if $\Delta'$ is constrained to be tame or wild. By Proposition
\ref{prop:XikvsXiktilde}, $\tilde\Xi_k(I)$ is a finite constant. Finally,
$\E_\Delta\tilde\Xi_k(L'(A+\epsilon\Delta)R)\ge \tilde\Xi_k(L')+
\tilde\Xi_k(R)+C_{\ref{prop:splitting}}\log\epsilon$.
Combining the inequalities, the result is proved.
\end{proof}

\begin{lem}\label{lem:logconvex}
Let $f(t)=\sum_{i=1}^n a_ie^{b_it}$ where $a_i>0$ for each $i$.
Then $f(t)$ is log-convex.
\end{lem}

\begin{proof}
We have $(\log f)'=f'/f$, so that $(\log f)''=(ff''-(f')^2)/f^2$.
Now
\begin{align*}
ff''-(f')^2&=
\sum_{i\ne j}a_ia_je^{(b_i+b_j)t}(b_j^2-b_ib_j)
+\sum_i a_i^2e^{2b_it}(b_i^2-b_i^2)\\
&=\sum_{i<j}a_ia_je^{(b_i+b_j)t}(b_i^2+b_j^2-2b_ib_j)\\
&\ge 0.
\end{align*}
\end{proof}

\begin{lem}\label{lem:convex1}
Let $V$ be a $k$-dimensional subspace of $H$ and let $\Pi_V$ be
the orthogonal projection onto $V$.
Then $f(s):=\Xi_k(\diagpar{e^s}\circ\Pi_V)$ is a convex function.
\end{lem}

\begin{proof}
We first prove that for $0<s<t$, $f(s)\le \frac st f(t)$.
To see this, let $v_1,\ldots,v_k$ be an orthogonal basis of $V$
such that $\diagpar{e^t}v_1,\ldots,\diagpar{e^t}v_k$ are orthogonal.
Then $f(s)\le\sum_{i=1}^k\log\|\diagpar{e^s}v_i\|$.
By Lemma \ref{lem:logconvex}, $s\mapsto \log\|\diagpar{e^s}v_i\|
=\frac 12\log(\sum_j e^{-2sj}{(v_i)_j}^2)$
is convex, so that $\log\|\diagpar{e^s}v_i\|\le \frac st
\log\|\diagpar{e^t}v_i\|$. Hence $f(s)\le \frac st f(t)$ as claimed.

Now if $0<a<b<c$, let $W=\diagpar{e^a}V$, let $s=b-a$ and $t=c-a$.
Let $\alpha=\Xi_k(\diagpar{e^a}\Pi_V)$.
Now we have $f(a)=\alpha$, $f(b)=\Xi_k(\diagpar{e^b}\Pi_V)
=\Xi_k(\diagpar{e^{b-a}}\diagpar{e^a}\Pi_V)
=\Xi_k(\diagpar{e^{b-a}}\Pi_W)+\Xi_k(\diagpar{e^a}\Pi_V)=
\alpha+\Xi_k(\diagpar{e^s}\Pi_W)$.
Similarly $f(c)=\alpha+\Xi_k(\diagpar{e^t}\Pi_W)$ and the result follows
from the above.
\end{proof}

\begin{lem}\label{lem:convex2}
Let $A$ be a Hilbert-Schmidt operator on $H$.
Then $g(s)\colon=\Xi_k(\diagpar{e^s}A)$ is a convex function.
Similarly $h(s)\colon=\Xi_k(A\diagpar{e^s})$ is convex.
\end{lem}

\begin{proof}
Let $0<a<b<c$.
Let $V$ be the $k$-dimensional space spanned by the top $k$ right singular vectors
of $\diagpar{e^b}A$ and $\Pi_V$ be the orthogonal projection onto $V$.
Let $W=A(V)$ and $\Pi_W$ be the orthogonal projection onto $W$.
Then we have $\Xi_k(\diagpar{e^t}A\Pi_V)=
\Xi_k(\diagpar{e^t}\Pi_W)+\Xi_k(A\Pi_V)$, the sum of a convex function
and a constant by Lemma \ref{lem:convex1}.
Now $g(b)=\Xi_k(\diagpar{e^b}A)=\Xi_k(\diagpar{e^b}A\Pi_V)
\le \frac{c-b}{c-a}\Xi_k(\diagpar{e^a}A\Pi_V)
+\frac{b-a}{c-a}\Xi_k(\diagpar{e^c}A\Pi_V)
\le \frac{c-b}{c-a}g(a)+\frac{b-a}{c-a}g(c)$
as required.

We have $h(s)=\Xi_k(A\diagpar{e^s})=\Xi_k(\diagpar{e^s}A^*)$,
which is convex by the above.
\end{proof}

\begin{prop}\label{prop:Xidiff}
Let $A$ be a Hilbert-Schmidt operator on $H$.
Then
$$
\Xi_k(\diagpar3A\diagpar3)-\Xi_k(A)\ge
\left(\frac{\log 3}{\log 2}\right)^2\big(\Xi_k(\diagpar2A\diagpar2)-\Xi_k(A)\big).
$$
\end{prop}

\begin{proof}
Let $f(s,t)=\Xi_k(\diagpar{e^s}A\diagpar{e^t})-\Xi_k(A)$.
Since $\diagpar a$ is contractive for $a>1$, we have $f(\log 3,0)\le 0$
and $f(0,\log 2)\le 0$.
Now Lemma \ref{lem:convex2} applied to $\Xi_k(\diagpar3A\diagpar{e^t})-\Xi_k(A)$
implies that $f(\log 3,\log 2)\le \frac{\log 2}{\log 3}f(\log 3,\log 3)$.
Applying the lemma to $\Xi_k(\diagpar{e^s}A\diagpar 2)-\Xi_k(A)$ implies
$$
f(\log 2,\log 2)\le \frac{\log 2}{\log 3}f(\log 3,\log 2)\le
\left(\frac{\log 2}{\log 3}\right)^2f(\log 3,\log 3),
$$
as required.
\end{proof}

\begin{lem}\label{lem:subaddcor}
Let $\sigma$ be an ergodic
measure-preserving transformation of $(\Sigma,\mathbb P)$.
Let $(f_n)$ be a sub-additive sequence of functions (that is
$f_{n+m}(\omega)\le f_n(\sigma^m\omega)+f_m(\omega)$ for each
$\omega\in\Omega$ and $n,m>0$) such that
$\inf_{n>0}\int \frac 1nf_n\,d\PP>-\infty$.
For any $\epsilon>0$, there exist $\chi>0$ and $n_0$ such that if
$M\ge n_0$ and $A$ is any set with $\PP(A)<\chi$
then $\int_A f_M\,d\PP > -\epsilon M$.
\end{lem}

\begin{proof}
Let $\alpha=\lim\int (f_n/n)\,d\PP$. Let $\epsilon>0$ be given.
Let $\chi$ be small enough that $\int_B f_1\,d\PP<\frac \epsilon3$ for any set
$B$ with $\PP(B)\le \chi$ and so that $2\chi(\alpha+\frac\epsilon3)
>-\frac\epsilon3$.
By the Kingman sub-additive ergodic theorem, there exists $m_0$ such that
for $M\ge m_0$,
$\PP(\{\omega\colon f_M(\omega)>(\alpha+\frac\epsilon3) M\})<\chi$.

Now let $A$ be an arbitrary set with $\PP(A)<\chi$.
We split $\Omega$ into three sets: $A$, $G=\{\omega\in A^c\colon
f_M(\omega)\le (\alpha+\frac \epsilon3) M\}$ and $B=A^c\setminus G$
(and note that $\PP(G^c)\le2\chi$).
Now we have
\begin{align*}
\alpha M&\le \int_\Omega f_M\,d\PP\\
&=\int_A f_M\,d\PP+\int_B f_M\,d\PP + \int_G f_M\,d\PP\\
&\le \int_A f_M\,d\PP+
\int_B (f_1+\ldots+f_1\circ\sigma^{M-1})\,d\PP
+(\alpha+\tfrac\epsilon3) M\PP(G).
\end{align*}
Hence we see
\begin{align*}
\int_A f_M\,d\PP&
\ge \alpha M-M\tfrac\epsilon3-(\alpha+\tfrac\epsilon3)M(1-\PP(G^c))\\
&=-\tfrac{2\epsilon}3M+(\alpha+\tfrac\epsilon3)M\PP(G^c)\ge -\epsilon M,
\end{align*}
as required.
\end{proof}

\begin{lem}\label{lem:trivial}
For all $k$, there exists a $C_{\ref{lem:trivial}}$ such that for any
bounded operator $A$ one has
$$
\Xi_k(A)\ge \tilde\Xi_k(A)-C_{\ref{lem:trivial}}.
$$
\end{lem}

\begin{proof}
We have $\tilde\Xi_k(A)=\E_{\Delta_1,\Delta_2}\Xi_k(\Pi_k\Delta_1 A\Delta_2\Pi_k)
\le 2\E\Xi_k(\Delta)+\Xi_k(A)\le 2k\E\log \|\Delta\|_\op+\Xi_k(A)$, where
we used sub-additivity of $\Xi_k$ for the first inequality and the
fact that $s_i(B)\le\|B\|_\op$ for the second. Hence it suffices
to show that $\E\log\|\Delta\|_\op<\infty$. But $\E\log\|\Delta\|_\op
\le \E\|\Delta\|_\op\le \E\|\Delta\|_\HS\le\sum_{i,j}\E|\Delta_{ij}|=
\sum_{i,j}3^{-(i+j)}\E|N|<\infty$.
\end{proof}

\section{Convergence of the Lyapunov exponents}\label{sec:puttingitalltogether}

\begin{proof}[Proof of Theorem \ref{thm:exponents}]
Rather than control the exponents directly, it is more straightforward, and clearly
equivalent, to control the partial sums of the exponents. Let $\mu_1(A)\ge
\mu_2(A)\ge\ldots$ denote the Lyapunov exponents of the cocycle $A$ listed
with multiplicity in decreasing order. We then let $\Lambda_k(A)=\mu_1(A)+
\ldots+\mu_k(A)$. We are aiming to show that $\Lambda_k(\pertA)\to
\Lambda_k(A)$ for each $k$. By an argument of Ledrappier and Young
\cite{LedrappierYoung},
explained slightly differently in our earlier paper \cite{FGTQ-cpam},
it suffices to show
that $\epsilon\mapsto\Lambda_k(\pertA)$ is upper semi-continuous for each $k$;
and lower semi-continuous for those $k$ such that $\mu_{k+1}(A)<\mu_k(A)$.

\subsection{Upper semi-continuity} We shall
show $\limsup_{\epsilon\to 0}\Lambda_k(\pertA)\le \Lambda_k(A)$.
To see this, let $\eta>0$. By the sub-additive ergodic theorem, there exists
an $n$ such that $\frac 1n\int\Xi_k(\cocycle n\om)\,d\PP(\om)<\Lambda_k(A)+\eta$.
As $\epsilon\to 0$, we have $\|\cocyclepss n\oom-\cocycle n\om\|\to 0$
and hence $\Xi_j(\cocyclepss n\oom)\to\Xi_j(\cocycle n\om)$ for all $\oom\in
\bar\Omega$. Set $g(\oom)=1+\|A(\omega_0)\|$ and $h(\oom)=\|\Delta_0\|$.
Then for $\epsilon<1$, $\log\|\cocyclepss n\oom\|\le\sum_{i=0}^{n-1}
\log(g+h)(\bar\sigma^i\oom)$. Since this is integrable, the Reverse Fatou Lemma
implies that $\limsup_{\epsilon\to 0}\frac 1n\int\Xi_j(\cocyclepss n\oom)
\,d\bar\PP(\oom)<\Lambda_k(A)+\eta$. Hence $\Lambda_k(\pertA)<
\Lambda_k(A)+\eta$ for sufficiently small $\epsilon$.

\subsection{Choice of Parameters}Now we move to showing the lower
semi-continuity of $\Lambda_k(\pertA)$
in the case where $\mu_{k+1}(A)<\mu_k(A)$. We assume without loss of generality
(by scaling the entire cocycle by a constant if necessary) that $\mu_{k+1}(A)<0<
\mu_k(A)$.

Let $\eta>0$.
We are seeking an $\epsilon_0$ such that for $\epsilon<\epsilon_0$,
$\Lambda_k(\pertA)>\Lambda_k(A)-\eta$.
First, choose an $n_0$ and $\chi$ such that the following inequalities are satisfied:
\begin{align*}
&\chi<
\min\left(
\frac{C_{\ref{lem:goodPert}}\eta}
{48\max(C_{\ref{prop:splitting}},C_{\ref{cor:twoglue}})},
\frac\eta{18\max(C_{\ref{prop:step2tilde}},C_{\ref{prop:badtriangleineq}}
(1+\frac2{C_{\ref{lem:goodPert}}}))}\right);\\
&\chi<\frac\eta{72k\int \log(1+\|A_\om\|_\SHS)\,d\PP(\om)};\\
&\int_B\Xi_k(\cocycle N\om)\,d\PP(\om)>-\frac{\eta N}{72}\text{
for $N\ge n_0$ if $\PP(B)<\chi$;}\\
&\int_B\log^+\|A_\om\|_\SHS\,d\PP(\om)
<\frac{\eta}{108k}\text{ if $\PP(B)<\chi$.}
\end{align*}
That $n_0$ and $\chi$ can be chosen to satisfy the third inequality is a consequence of
Lemma \ref{lem:subaddcor}. Let $\delta$ be chosen so that $\PP(G^c)<\chi/2$,
where $G$ is the event that the block $\cocycle N\om$ is good
as in Lemma \ref{lem:goodBlocks}. Let $\epsilon_1$ be chosen
so that $N_\epsilon:=\lfloor C_{\ref{lem:goodPert}}|\log\epsilon|\rfloor>n_0$ for all
$\epsilon<\epsilon_1$.
Let $\epsilon_2$ be such that the probability that an $N_\epsilon$-block of
$\Delta$'s contains a wild perturbation
is less than $\chi/2$ for all $\epsilon<\epsilon_2$
(such an $\epsilon_2$ exists by Lemma \ref{lem:goodPert}).
Let $\bar{G}=\{ \oom \in \bar \Om: \om \in G; \Delta_0, \dots, \Delta_{N-1} \text{ are tame}\}$.
We will only consider $\epsilon$'s that are smaller than $\epsilon_1$ and $\epsilon_2$
for the remainder of the argument. In particular $\bar{\PP}(\bar G ^c)<\chi$.

We need to control $\E\Xi_k(\cocycle{nN}\oom)$, where $N$ is the length
of a block (as given by Lemma \ref{lem:goodPert}),
 and we let $n\to\infty$. Here and below, the superscript $\epsilon$
indicates that we are studying the perturbed cocycle.

\subsection{Replacing $\Xi_k$ with $\tilde\Xi_k$}We have
\begin{equation}
\Xi_k(\cocyclepss{nN}\oom)
\ge\; \tilde\Xi_k(\cocyclepss{nN}{\oom})-C_{\ref{lem:trivial}},
\label{eq:XigeXitildenew}
\end{equation}
by Lemma \ref{lem:trivial}. The advantage of $\tilde\Xi_k$ over $\Xi_k$ is
that it admits a lower bound in terms of sub-blocks.

\subsection{Splitting $\cocyclepss{nN}\oom$ into good and bad blocks}
Recall a block $\cocyclepss N{\bar\sigma^{jN}\oom}$ is said to be good if
$\sigma^{jN}\oom\in \bar G$, that is the unperturbed cocycle is
well-behaved, and the perturbations are tame. Given $\oom$,
we split up $\cocyclepss{nN}\oom$ into blocks of length $N$.
Whenever three or more consecutive blocks are good, we form a
\emph{super-block}, $G^\epsilon$,
consisting of the concatenation of the good blocks other than
the first and last good blocks. All of the remaining blocks are called \emph{filler} blocks.
The $B^\epsilon$ are the filler blocks stripped of their first and last matrices.

We have
\begin{equation}\label{eq:blocksplitnew}
\begin{split}
&\E\Big(\tilde\Xi_k(\cocyclepss{nN}\oom)\Big)\ge\\
& \E\Big(\tilde\Xi_k(B^\epsilon)+\tilde\Xi_k(G^\epsilon) + \tilde\Xi_k(B^\epsilon)
+\tilde\Xi_k(B^\epsilon) + \tilde\Xi_k(B^\epsilon)+\ldots\Big)-E_1,
\end{split}
\end{equation}
where the splitting in the last line is into super-blocks (of variable length, all
a multiple of $N$), here designated by $G^\epsilon$,
and filler blocks, $B^\epsilon$, all of length $N-2$ and $E_1$ denotes an expected error term
that we now estimate.

To obtain \eqref{eq:blocksplitnew}, we split the concatenation
of $n$ blocks of length $N$ into
the super-blocks and filler blocks as described  above by repeatedly
applying Proposition \ref{prop:splitting},
which sacrifices a single matrix as `glue' at each splitting
site (or Corollary \ref{cor:twoglue}
in the case of two consecutive filler blocks when two matrices are sacrificed).
Since the expected number of non-good $N$-blocks is less than $\chi n$
and each such block gives rise to at most 4 transitions between 
adjacent blocks in the concatenation (the worst case happens when two
super-blocks are joined by three fillers),
we deduce $E_1\le 4\chi n\max(C_{\ref{prop:splitting}},
C_{\ref{cor:twoglue}})|\log\epsilon|$.
From Lemma \ref{lem:goodPert}, $|\log\epsilon|\le 2N/C_{\ref{lem:goodPert}}$, so that
\begin{equation}\label{eq:E1est}
E_1\le 8\chi nN\max(C_{\ref{prop:splitting}},C_{\ref{cor:twoglue}})/
C_{\ref{lem:goodPert}}
\le \tfrac16\eta nN.
\end{equation}

\subsection{Comparison of $\tilde\Xi_k(G^\ep)$ and $\Xi_k(G^\ep)$}
To bound one of the $\tilde\Xi_k(G^\epsilon)$, the contribution
from one of the super-blocks,
we first compare to $\Xi_k(G^\epsilon)$, the corresponding contribution to the
genuine singular values; and then compare to $\Xi_k(G^0)$, the singular values
of the unperturbed block. Recall that each $G^\epsilon$ is preceded
by an $N$-block $L^\epsilon$
and followed by an $N$-block $R^\epsilon$ such that the  enlarged
block $L^\epsilon G^\epsilon
R^\epsilon$ consists entirely of good blocks.

For the first comparison, we have
\begin{equation}\label{eq:conv}
\begin{split}
\tilde\Xi_k(G^\epsilon)&\ge \Xi_k(\diag_3G^\epsilon\diag_3)-
C_{\ref{prop:XikvsXiktilde}}\\
&\ge \Xi_k(G^\epsilon)+3(\Xi_k(\diag_2G^\epsilon\diag_2)-\Xi_k(G^\epsilon))-
C_{\ref{prop:XikvsXiktilde}},
\end{split}
\end{equation}
using Propositions \ref{prop:XikvsXiktilde} and \ref{prop:Xidiff} respectively.
Now
\begin{equation}\label{eq:subadditivity}
\Xi_k(\diag_2G^\epsilon\diag_2)\ge
\Xi_k(L^\epsilon G^\epsilon R^\epsilon)-\Xi_k(L^\epsilon\diag_2^{-1})-
\Xi_k(\diag_2^{-1}R^\epsilon)
\end{equation}
by sub-additivity, and
\begin{equation}\label{eq:goodness}
\begin{split}
&\Xi_k(L^\epsilon G^\epsilon R^\epsilon)
\ge
\log|\det(L^\epsilon G^\epsilon R^\epsilon|_{F^{\perp}(R_0)})|\\
&=
\log|\det(L^\epsilon|_{G^\epsilon R^\epsilon(F^\perp(R_0))})|
+
\log|\det(G^\epsilon|_{R^\epsilon(F^\perp(R_0))})|\\
&\quad+
\log|\det(R^\epsilon|_{F^\perp(R_0)})|\\
&
\ge
\Xi_k(L^\epsilon)+\Xi_k(G^\epsilon)+\Xi_k(R^\epsilon)+3k\log\delta,
\end{split}
\end{equation}
where we made use of Proposition \ref{prop:almostXi} for the second inequality
(Lemmas~\ref{lem:ly}\eqref{it:contr} and \ref{lem:goodBlocks}\eqref{it:sep}, \eqref{it:Econt}
and \eqref{it:Fcont} were used to ensure the hypotheses of that Proposition are satisfied).
Combining inequalities \eqref{eq:conv}, \eqref{eq:subadditivity} and
\eqref{eq:goodness}, we obtain
\begin{equation*}
\begin{split}
\tilde\Xi_k(G^\epsilon)\ge \Xi_k(G^\epsilon)+3\Big(&\Xi_k(L^\epsilon)+\Xi_k(R^\epsilon)
-\Xi_k(L^\epsilon\diag_2^{-1})\\
&-\Xi_k(\diag_2^{-1}R^\epsilon)+3k\log\delta\Big)-C_{\ref{prop:XikvsXiktilde}}.
\end{split}
\end{equation*}
By Lemmas~\ref{lem:sv}\eqref{it:sjcont}, \ref{lem:goodBlocks}\eqref{it:svsep} 
and \ref{lem:goodPert}, we have $\Xi_k(L^\epsilon)$
and $\Xi_k(R^\epsilon)$ are non-negative.
By Lemma \ref{lem:goodBlocks}\eqref{it:logSHS}, using sub-additivity, we have
$\Xi_k(L^\epsilon\diagpar2^{-1}),\Xi_k(\diagpar2^{-1}R^\epsilon)
\le 2kN\int \log(1+\|A_\om\|_\SHS)\,d\PP(\om)$.
Hence for each good block, we have
\begin{equation*}
\tilde\Xi_k(G^\epsilon)
\ge \Xi_k(G^\epsilon)-\eta N/(6\chi)
+9k\log\delta-C_{\ref{prop:XikvsXiktilde}}.
\end{equation*}

\subsection{Comparison of $\Xi_k(G^\ep)$ and $\Xi_k(G^0)$}
Next, by Proposition \ref{prop:GlueingGoodBlocks}, we have
$\Xi_k(G^\epsilon)\ge \Xi_k(G^0)+2k\ell\log\delta$,
where $\ell$ is the number of blocks forming the $G^\epsilon$ super-block, so that
overall, for each good block, we have
\begin{equation}\label{eq:goodcomp}
\tilde\Xi_k(G^\epsilon)\ge \Xi_k(G^0)-\eta N/(6\chi)
+11k\ell\log\delta-C_{\ref{prop:XikvsXiktilde}},
\end{equation}
where $G^0$ is the corresponding unperturbed block.

In summary,
\begin{equation}\label{eq:blocksplitnew2}
\begin{split}
&\E\Big(\tilde\Xi_k(A_\omega^{\epsilon(nN)})\Big)\ge\\
& \E\Big(\tilde\Xi_k(B^\epsilon)+\Xi_k(G^0) + \tilde\Xi_k(B^\epsilon)
+\tilde\Xi_k(B^\epsilon) + \tilde\Xi_k(B^\epsilon)+\ldots\Big)-E_1-E_2,
\end{split}
\end{equation}
where $E_2$ is the combined contribution of  the errors coming from 
good blocks via \eqref{eq:goodcomp}.

\subsection{Comparison of $\E\tilde\Xi_k(B^\ep)$ and $\tilde\Xi_k(B^0)$}
We next work on giving a lower bound for the terms of the form
$\E\tilde\Xi_k(B^\epsilon)$. It turns out to be
convenient to bound this in the opposite order than the way we obtained bounds for
$\E\tilde\Xi_k(G^\epsilon)$. Namely, we show $\E\tilde\Xi_k(B^\epsilon)
\gtrsim \tilde\Xi_k(B^0)\gtrsim \Xi_k(B^0)$.

If the filler block $B^\ep=\cocyclepss{N-2}{\bar\sigma^{jN+1}\oom}$
 is not type II bad, we have $\E\tilde\Xi_k(B^\epsilon)
\ge \tilde\Xi_k(B^0)-C_{\ref{prop:step2tilde}}N$ by Proposition
\ref{prop:step2tilde}, where $B^0=\cocycle{N-2}{\sigma^{jN+1}\om}$, the unperturbed block.
When $B^\ep$ is type II
bad, we have $\E\tilde\Xi_k(B^\epsilon)\ge \tilde\Xi_k(B^0)+
C_{\ref{prop:badtriangleineq}}(\log\epsilon-N)$ by Proposition~\ref{prop:badtriangleineq}. Since
by Lemma \ref{lem:goodPert}, we have $\log\epsilon>-2N/C_{\ref{lem:goodPert}}$,
we get $\E\tilde\Xi_k(B^\epsilon)\ge \tilde\Xi_k(B^0)-
C_{\ref{prop:badtriangleineq}}N(1+2/C_{\ref{lem:goodPert}})$ in this case.
We therefore have in either case that
\begin{equation}\label{eq:xitbepsvsxitb0}
\E\tilde\Xi_k(B^\epsilon)\ge \tilde\Xi_k(B^0)-\eta/(18\chi)N,
\end{equation}

\subsection{Comparison of $\tilde\Xi_k(B^0)$ and $\Xi_k(B^0)$}
For the estimate $\tilde\Xi_k(B^0)\gtrsim\Xi_k(B^0)$, we use an argument
similar to that in \eqref{eq:conv} and \eqref{eq:subadditivity} above.
Namely, let the matrices preceding and following $B^0$ in the unperturbed
cocycle be $L^0$ and $R^0$. We also write $\bar B^0= \cocycle N{\sigma^{jN}\om}$
for the $N$-block, $L^0B^0R^0$.
Then as before, we have
\begin{equation}\label{eq:badcomp}
\begin{split}
&\tilde\Xi_k(B^0)\ge \Xi(B^0)+3(\Xi_k(\diagpar 2B^0\diagpar2)-\Xi_k(B^0)) -
C_{\ref{prop:XikvsXiktilde}}\\
&\ge \Xi_k(B^0)+3\big(\Xi_k(\bar B_0)-\Xi_k(L^0\diag_2^{-1})
-\Xi_k(\diag_2^{-1}R^0)-\Xi_k(B^0)\big)- C_{\ref{prop:XikvsXiktilde}}\\
&= \Xi_k(\bar B^0)+2\big(\Xi_k(\bar B^0)-\Xi_k(B^0)\big)-3\big(\Xi_k(L^0\diag_2^{-1})
+\Xi_k(\diag_2^{-1}R^0)\big)- C_{\ref{prop:XikvsXiktilde}}.
\end{split}
\end{equation}
We have the estimate for the subtracted terms in \eqref{eq:badcomp}:
\begin{equation*}
2\Xi_k(B^0)+3(\Xi_k(L^0\diagpar2^{-1})+\Xi_k(\diagpar2^{-1}R^0))
\le 3kF(\sigma^{jN}\omega),
\end{equation*}
where $F(\omega)=\sum_{i=0}^{N-1}\log^+\|A(\sigma^i\omega)\|_\SHS$.
This is a consequence of sub-additivity of $\Xi_k$, the fact that 
$\|A\diagpar2^{-1}\|_{\op}, \|A\|_{\op}\leq \|A\|_{\SHS}$ for every 
$A\in \SHS$ and $\Xi_k(A) \leq k \log \|A\|_{\op}$.
By the choice of $\chi$, we have
$\int_{\bar G^c}F(\omega)\,d\bar\PP(\oom)<\eta N/(108k)$.
The combined contribution from the subtracted terms in \eqref{eq:badcomp}
to all of the $\tilde\Xi_k(B^\epsilon)$ terms in \eqref{eq:blocksplitnew}
is bounded above by
$$
3k\sum_{j=0}^{n-1}
\mathbf 1_{\bad}(\bar\sigma^{jN}\oom)
F(\sigma^{jN}\om),
$$
where $\bad$ is $\bar G^c\cup \bar\sigma^{-N}\bar G^c\cup \bar\sigma^N \bar G^c$,
the set of points which are the first index of a filler block.
Hence the expectation of the contribution of the subtracted terms in
\eqref{eq:badcomp} is at most $\eta nN/12$.

We use a similar argument to give a lower bound for the sum of the
added $2\Xi_k(\bar B_0)$ terms in \eqref{eq:badcomp}. These terms
are
\begin{equation}\label{eq:fillerAdd}
2\sum_{j=0}^{n-1}
\mathbf 1_\bad(\bar\sigma^{jN}\oom)
\Xi_k(\cocycle N{\sigma^{jN}\om}).
\end{equation}
By the choice of $\chi$, $\int_B\Xi_k(\cocycle N\om)\ge -\eta N/72$ for any
set, $B$, of measure at most $\chi$.
Hence, the expected value of the expression in \eqref{eq:fillerAdd} is bounded below by
$-\eta nN/12$.

Combining these estimates along all filler blocks occuring in \eqref{eq:blocksplitnew}, we see
\begin{equation}\label{eq:xitildevsxi}
\E\left(\sum_{j=0}^{n-1}\mathbf 1_\bad(\bar\sigma^{jN}\oom)
\big(\tilde \Xi_k(\cocycle{N-2}{\sigma^{jN+1}\om})-
\Xi_k(\cocycle N{\sigma^{jN}\om})\big)\right)
\ge -\eta nN/6.
\end{equation}

\subsection{Combining the inequalities}
At this point, we have (combining inequalities
\eqref{eq:blocksplitnew}, \eqref{eq:goodcomp}, \eqref{eq:xitbepsvsxitb0}
and \eqref{eq:xitildevsxi}),
\begin{equation}\label{eq:blocksplitnew3}
\begin{split}
&\E\Big(\tilde\Xi_k(\cocyclepss{nN}{\oom})\Big)\ge\\
& \E\Big(\Xi_k(\bar B^0)+\Xi_k(G^0) + \Xi_k(\bar B^0)
+\Xi_k(\bar B^0) + \Xi_k(\bar B^0)+\ldots\Big)-E_1-E_2-E_3,
\end{split}
\end{equation}
where $E_3$ comes from the contributions of \eqref{eq:xitbepsvsxitb0} and \eqref{eq:xitildevsxi}.
Then using \eqref{eq:XigeXitildenew},
\begin{align*}
\E&\Xi_k(\cocyclepss{nN}\oom)\\
\ge\,&
\E\big(\Xi_k(\bar B^0)+\Xi_k(G^0)+\Xi_k(\bar B^0)+\Xi_k(\bar B^0)+
\Xi_k(\bar B^0)+\ldots\big) -C_{\ref{lem:trivial}} \\
&-\Big(\tfrac 16\eta nN \Big)- \Big(\tfrac 16(\eta N/\chi) \E n_\textsf{Super}+
C_{\ref{prop:XikvsXiktilde}}\E n_\textsf{Super} -11kn\log\delta\Big)\\
&- \Big(\tfrac{1}{18}(\eta N/\chi)\E n_\bad+\tfrac 16\eta nN\Big),
\end{align*}
where $n_\bad$ and $n_\textsf{Super}$ are the number of filler and super-blocks
respectively in $\cocyclepss{nN}\oom$. By sub-additivity, the first term in parentheses
is at least $\E\Xi_k(\cocycle{nN}\om)$.
We have $\E n_\bad<3\chi n$ and $\E n_\textsf{Super}<\chi n$,
\begin{align*}
\E\Xi_k(\cocyclepss{nN}\oom)\ge
&\E \Xi_k(\cocycle{nN}\om)-C_{\ref{lem:trivial}}-\tfrac 23\eta nN
-C_{\ref{prop:XikvsXiktilde}}\chi n +11kn\log\delta.
\end{align*}
As $\epsilon$ is reduced to 0, $\delta$ does not grow, but $N\to\infty$
so that for sufficiently small $\ep$,
we have
$$
\E\Xi_k(\cocyclepss{nN}\oom)\ge \E \Xi_k(\cocycle{nN}\om)-\eta nN.
$$
Hence we deduce $\Lambda_k(\pertA)\ge \Lambda_k(A)-\eta$, as required.
\end{proof}

\section{Convergence of the Oseledets spaces}\label{S:oselConv}

\begin{proof}[Proof of Theorem~\ref{thm:spaces}]

Let $k=D_i$ be as in the statement of the theorem.
Let us assume, by possibly rescaling the cocycle by a constant, that $\mu_k>0>\mu_{k+1}$.
Let $\delta_0<1$ and
$$
U_\epsilon=\big\{\oom\colon
\angle\big(\topspaceDSpert k{\oom},
\topspaceDSunpert k{\om}\big)>2\delta_0\big\}.
$$
We will show that for every $0<\eta<1$ and every sufficiently small
$\ep>0$, $\bar{\PP}(U_\ep)<\eta$.

Once this is established, convergence in probability  of the Oseledets spaces
$Y_k^\ep(\oom)$ to $Y_k^0(\om)$ follows via the identity
$Y_k^\ep(\oom)= \topspaceDSpert k{\oom} \cap \bottomspaceDSpert {k-1}{\oom}$,
and the fact that $\bottomspaceDSpert {k-1}{\oom}$ coincides with the
orthogonal complement of the top $k$-dimensional Oseledets space of the
adjoint cocycle $(\pertA)^*$, which converges in probability by the same argument.
See~\cite[\S4]{FGTQ-cpam}  for details.

In what follows, we will repeatedly apply Lemma~\ref{lem:goodBlocks},
assuming $\xi<\frac{\eta}{3}, \delta_1<\min \{ \delta_0, \frac{\mu_k \eta}{15k}\}$,
and so the value of $\tau$ provided by Lemma~\ref{lem:goodBlocks} satisfies
 $\tau\leq \del_1\leq \frac{\mu_k \eta}{15k}$.
The corresponding value of $\delta<\delta_0$ provided by
Lemma~\ref{lem:goodBlocks} will also be used in the application of Lemma~\ref{lem:ly}.

Let
$$
W_\epsilon=
\bar\sig^{-2N} U_\ep\cap \bar{G} \cap \bar \sig^{-N}\bar G,
$$
where $N$ depends on $\ep$ as in Lemma~\ref{lem:goodPert}.
For sufficiently small $\ep$, we have
$\bar{\PP}(\bar{G} \cap \bar \sig^{-N}\bar G) \geq 1-\frac{2\eta}{3}$, so that once we show
$\bar{\PP}(W_\ep)< \frac{\eta}{3}$, we will be able to conclude that
$\bar{\PP}(U_\ep)= \bar{\PP}(\sig^{-2N}U_\ep)\leq \bar{\PP}(W_\ep) +
\bar{\PP}(\bar{G}^c \cup \bar \sig^{-N}\bar G^c)<\eta$.

\begin{lem}\label{lem:bad1}
Suppose $\oom \in W_\epsilon$. Then
${\perp}(\topspaceDSpert k{\bar\sig^N\oom},
\bottomspacemat k{\cocycle N{\sig^N\om}})\le \delta$.
\end{lem}

\begin{proof}
We prove the contrapositive:
Suppose ${\perp}(\topspaceDSpert k{\bar\sig^N\oom},
\bottomspacemat k{\cocycle N{\sig^N\om}})>\delta$.
Applying Lemma \ref{lem:ly}\eqref{it:contr} (using $\sig^N\oom\in\bar G$
and setting $V=\topspaceDSpert k{\bar\sig^N\oom}$),
we see $\angle(\topspaceDSpert k{\bar\sig^{2N}\oom},
\topspacemat k{\cocycle N{\sig^N\om}})<\delta$.
As $\angle(\topspacemat k{\cocycle N{\sig^N\om}},
\topspaceDSunpert k{\sig^{2N}\om})<\delta$ by Lemma
\ref{lem:goodBlocks}\eqref{it:Econt}, we deduce the bound
$\angle(\topspaceDSpert k{\bar\sig^{2N}\oom},
\topspaceDSunpert k{\cocycle N{\sig^{2N}\om}})<2\delta$,
which contradicts $\oom\in \bar\sig^{-2N}U_\ep$.
\end{proof}

\begin{lem}\label{lem:smallPerp}
If $\oom, \bar\sig^N\oom \in \bar G$ and ${\perp}(\topspaceDSpert k{\bar\sig^N\oom},
\bottomspacemat k{\cocycle N{\sig^N\om}})<\delta$,
then ${\perp}(\topspaceDSpert k\oom,
\bottomspacemat k{\cocyclepss N\oom})<\delta^{-1} e^{-(\mu_k-\tau) N}$.
\end{lem}
\begin{proof}
We will show the contrapositive.
Assume ${\perp}(\topspaceDSpert k\oom,
\bottomspacemat k{\cocyclepss N\oom})\ge\delta^{-1} e^{-(\mu_k-\tau) N}$.
Let  $v\in \topspaceDSpert k\oom$ be of length 1.
Let $v=u+w$, with $u \in \bottomspacemat k{\cocyclepss N\oom}$
and $w\in \bottomspacemat k{\cocyclepss N\oom}^\perp$.
Then, $\|w\| \geq \delta^{-1} e^{-(\mu_k-\tau) N}$,
and so $\|\cocyclepss N\oom w \| \geq \delta^{-1}$.
Also, $\|\cocyclepss N\oom u \|\leq 2$ by Lemma \ref{lem:ly}\eqref{it:pertsvs}.
Normalizing $\cocyclepss N\oom v$ and recalling that
$\cocyclepss N\oom \topspaceDSpert k\oom= \topspaceDSpert k{\bar\sig^N\oom}$
and $\cocyclepss N\oom\bottomspacemat k{\cocyclepss N\oom}^\perp=
\topspacemat k{\cocyclepss N\oom}$,
we obtain that a point of $\topspaceDSpert k{\bar\sig^N\oom}\cap S$
is within $4\delta$ of $\topspacemat k{\cocyclepss N{\om}}\cap S$. Hence,
by Lemma \ref{lem:symmcloseness},
$\angle( \topspaceDSpert k{\bar{\sig}^N\oom},
\topspacemat k{\cocyclepss N\oom})  < 4\delta$.

By Lemmas~\ref{lem:ly}\eqref{it:spacecont} and \ref{lem:goodBlocks}\eqref{it:Econt},
$\angle(\topspacemat k{\cocyclepss N\oom} ,  \topspaceDSunpert k{{\sig^N\om}})<  2\delta$.
Hence,  $\angle(\topspaceDSpert k{\bar{\sig}^N\oom} ,
\topspaceDSunpert k{{\sig^N\om}})<  6\delta$.
By Lemma \ref{lem:goodBlocks}\eqref{it:Fcont},
$\angle(\bottomspacemat k{\cocycle N\om},
\bottomspaceDSunpert k{{\sig^N\om}})<  \delta$.
Lemma~\ref{lem:goodBlocks}\eqref{it:sep} ensures that
${\perp}(\topspaceDSunpert k{\sig^N\om}, \bottomspaceDSunpert k{{\sig^N\om}})>10\delta$,
and combining with the above, we conclude that
${\perp}(\topspaceDSpert k{\bar\sig^N\oom},
\bottomspacemat k{\cocycle N{\sig^N\om}})> 3\delta$.
\end{proof}

\begin{lem}\label{lem:upperBdXikGood}
If $\ep$ is sufficiently small that $\delta^{-1}+2<e^{k\tau N}$,
 $\oom\in \bar G$ and ${\perp}(\topspaceDSpert k\oom,
\bottomspacemat k{\cocyclepss N\oom})<\delta^{-1} e^{-(\mu_k-\tau) N}$, we have
$$
\Xi_k(\cocyclepss N\oom\vert_{\topspaceDSpert k\oom})
\le {(\mu_1+\ldots+\mu_{k-1}+ 2k\tau)N}.
$$
\end{lem}

\begin{proof}
By hypothesis, there exists a unit
length $v\in \topspaceDSpert k\oom$
such that $v=f+f^\perp$, with $f \in \bottomspacemat k{\cocyclepss N\oom},
f^\perp \in \bottomspacemat k{\cocyclepss N\oom} ^\perp$ and
$\| f^\perp\| < \delta^{-1} e^{-(\mu_k-\tau)N}$.

Now, since $\topspaceDSpert k\oom$ is $k$-dimensional,
$\Xi_k(\cocyclepss N\oom\vert_{\topspaceDSpert k\oom})$ is the
logarithm of the volume growth of any $k$-dimensional
parallelepiped in $\topspaceDSpert k\oom$ under $\cocyclepss N\oom$.
Let $v, v_2, \dots, v_k$ be an orthonormal basis
for $\topspaceDSpert k\oom$.
Then,
\begin{align*}
\vol(\cocyclepss N\oom v, \cocyclepss N\oom v_2,\dots, \cocyclepss N\oom v_k) &\leq
\vol(\cocyclepss N\oom f,\cocyclepss N\oom v_2,  \dots, \cocyclepss N\oom v_k) \\
&+
\vol(\cocyclepss N\oom f^\perp,\cocyclepss N\oom v_2, \dots, \cocyclepss N\oom v_k).
\end{align*}

By the choice of $f$, and using Lemma~\ref{lem:ly}\eqref{it:pertsvs},
\begin{align*}
\vol(\cocyclepss N\oom f,\cocyclepss N\oom v_2,  \dots, \cocyclepss N\oom v_k) &\leq
\|\cocyclepss N\oom f\| e^{\Xi_{k-1}(\cocyclepss N\oom)} \\
& \leq 2 e^{(\mu_1+\ldots+\mu_{k-1} + (k-1)\tau)N}.
\end{align*}
Since $\| f^\perp\| < \delta^{-1} e^ {-(\mu_k-\tau)N}$, then
$\vol(\cocyclepss N\oom f^\perp,\cocyclepss N\oom v_2, \dots, \cocyclepss N\oom v_k)\leq
\|f^\perp\| e^{\Xi_k(\cocyclepss N\oom)}
<  \delta^{-1}e^{(\mu_1+\ldots+\mu_{k-1} + k\tau)N}$.
\end{proof}

\begin{lem}\label{lem:upperBdXik}
There exists $\ep_0>0$ and $M\in \N$ such that for every $\ep<\ep_0, N\geq M$
and $B\subset \bar{\Om}$, we have that
\[
\int_B \Xi_k(\cocyclepss N\oom)  d\bar{\PP} < N (\mu_1+\dots +\mu_k) \bar{\PP}(B) + 2 \tau N.
\]
In particular, for all sufficiently small $\ep$, the above holds for $N$
chosen as in Lemma~\ref{lem:goodPert}.
\end{lem}
\begin{proof}
By the $L^1$ convergence in the sub-additive ergodic theorem, there exists
$M>0$ be such that
$\| \Xi_k(\cocycle n\om) - n(\mu_1+\dots + \mu_k)  \|_1\leq n{\tau}$ for every
$n\geq M$. In particular,  for every $n\geq M$ and every $B\subset \bar{\Om}$,
\[
\int_B \Xi_k(\cocycle n\om)  d\bar{\PP} < n(\mu_1+\dots + \mu_k)\bar{\PP}(B)+ n\tau.
\]
Notice that $\Xi_k(\cocyclepss n\oom)\leq k \log^+ \| \cocyclepss n\oom\|_{\op}
\leq k\sum_{j=0}^{n-1}  (\log^+ \| A_{\sig^j\om}\|_{\op} + \ep \|\Delta_j\|_{\op})$,
where we have used the fact that $\log^+(x+y) \leq \log^+ (x) + |y|$.
For a fixed $n$, this shows that the family of functions $g_\ep(\oom)=\Xi_k(\cocyclepss n\oom)$
for $0\le \ep<1$ is dominated, and converges as $\ep\to 0$ to $\Xi_k(\cocycle n\om)$.
Hence, by the reverse Fatou lemma, for sufficiently small $\ep>0$, $n\in \{M, \dots, 2M-1\}$
and every $B\subset \bar{\Om}$,
\[
\int_B \Xi_k(\cocyclepss n\oom)  d\bar{\PP} <
n (\mu_1+\dots +\mu_k) \bar{\PP}(B) + 2 \tau n.
\]
Using sub-additivity of $\Xi_k$, we conclude that for every $N\geq M$,
and every $B\subset \bar{\Om}$,
\[
\int_B \Xi_k(\cocyclepss N\oom)  d\bar{\PP} <
N (\mu_1+\dots +\mu_k) \bar{\PP}(B) + 2\tau N.
\]
\end{proof}

Notice that if $\oom\in W_\ep$, then by Lemmas~\ref{lem:bad1}, \ref{lem:smallPerp} and
\ref{lem:upperBdXikGood} (each lemma establishing the hypothesis of the next one),
then $\Xi_k(\cocyclepss N\oom|_{\topspaceDSpert k\oom})
\le (\mu_1+\ldots+\mu_{k-1}+2k\tau)N$.
Combining this with Lemma \ref{lem:upperBdXik}, we see
\begin{align*}
\mu_1^\ep&+\ldots+\mu_k^\ep=
\lim_{n\to\infty}\frac 1n\int \Xi_k(\cocyclepss n\oom\vert_
{\topspaceDSpert k{\oom}})\,d\bar\PP(\oom)\\
&\le\frac 1N\int_{W_\ep}\Xi_k(\cocyclepss N\oom\vert_{
\topspaceDSpert k{\oom}})\,d\bar\PP(\oom)
+\frac1N\int_{W_\ep^c}\Xi_k(\cocyclepss N\oom)\,d\bar\PP(\oom)\\
&\leq
(\mu_1+\ldots+\mu_{k-1}+ 2k\tau) \bar{\PP}(W_\ep) +
(\mu_1+\dots +\mu_k) \bar{\PP}(W_\ep^c) + 2 \tau.
\end{align*}
Hence,
\[
\mu_k \bar{\PP}(W_\ep) \leq (\mu_1+\ldots+\mu_k) - 
(\mu_1^\ep+\ldots+\mu_k^\ep) + 4k \tau.
\]
In particular, in view of the convergence of the exponents,
for all sufficiently small $\epsilon$, we have $\bar\PP(W_\ep)
\le 5k\tau/\mu_k<\frac{\eta}{3}$.
\end{proof}

\section*{Acknowledgements}
GF and AQ acknowledge partial support from the Australian Research Council 
(DP150100017). The research of CGT has been supported by an 
ARC DECRA (DE160100147). AQ acknowledges the support of NSERC.
The authors are grateful to the the School of Mathematics and Statistics 
at the University of New South Wales, the School of Mathematics and 
Physics at the University of Queensland and the Department of Mathematics and 
Statistics at the University of Victoria for their hospitality, allowing for research 
collaborations which led to this project.
\bibliographystyle{abbrv}
\bibliography{volume_refs}
\end{document}